\documentclass{article}

\usepackage[left=1in, right=1in]{geometry}
\usepackage{graphicx}
\usepackage{amsmath, amssymb, amsthm}
\usepackage{enumerate}
\usepackage{hyperref}
\usepackage{url}
\usepackage{color}
\usepackage{authblk}

\makeatletter
\def\maketag@@@#1{\hbox{\m@th\normalfont\normalsize#1}}
\makeatother

\newcommand{\R}{\ensuremath{\mathbb{R}}}

\newcommand{\X}{\ensuremath{\mathbb{X}}}
\newcommand{\A}{\ensuremath{\mathbb{A}}}
\newcommand{\E}{\ensuremath{\mathbb{E}}}
\newcommand{\F}{\ensuremath{\mathbb{F}}}

\newcommand{\B}{\ensuremath{\mathcal{B}}}
\newcommand{\Gr}{\ensuremath{\mbox{Gr}}}

\newcommand{\U}{\ensuremath{\mathcal{U}}}

\newtheorem{thm}{Theorem}
\newtheorem{prp}[thm]{Proposition}
\newtheorem{lem}[thm]{Lemma}
\newtheorem{cor}[thm]{Corollary}

\newtheorem*{asmpT}{Assumption~T}
\newtheorem*{asmpHT}{Assumption~HT}
\newtheorem*{infComp}{Compactness Conditions}

\newtheorem*{asmpS}{Assumption~S}
\newtheorem*{asmpTprime}{Assumption~T'}
\newtheorem*{asmpHTprime}{Assumption~HT'}

\theoremstyle{definition}

\newtheorem{rem}{Remark}

\DeclareMathOperator*{\argmin}{arg\,min}

\begin{document}

\title{\textbf{On the Reduction of Total-Cost and Average-Cost MDPs to
    Discounted MDPs}}

\author[1]{Eugene A. Feinberg}
\author[2]{Jefferson Huang}

\affil[1]{\small Department of Applied Mathematics and Statistics \authorcr Stony Brook University, Stony Brook, NY, 11794-3600, USA \vspace{2ex}}
\affil[2]{\small School of Operations Research and Information Engineering \authorcr Cornell University, Ithaca, NY, 14853-3801, USA}

\maketitle

\begin{abstract}
  This paper provides conditions under which total-cost and
  average-cost Markov decision processes (MDPs) can be reduced to
  discounted ones. Results are given for transient total-cost MDPs
  with transition rates whose values may be greater than one, as well
  as for average-cost MDPs with transition probabilities satisfying
  the condition that there is a state such that the expected time to
  reach it is uniformly bounded for all initial states and stationary
  policies. In particular, these reductions imply sufficient
  conditions for the validity of optimality equations and the
  existence of stationary optimal policies for MDPs with undiscounted
  total cost and average-cost criteria. When the state and action sets
  are finite, these reductions lead to linear programming formulations
  and complexity estimates for MDPs under the aforementioned criteria.
\end{abstract}


\section{Introduction}
\label{intro}

This paper deals with the reduction of undiscounted total-cost and
average-cost Markov decision processes (MDPs) to discounted MDPs. For
undiscounted total costs, we consider a weighted-norm version of the
\textit{transient} case introduced by Veinott~\cite{veinott1969} in
the context of finite state and action sets and by
Pliska~\cite{pliska1978} in the context of Borel state and action
spaces. A feature of such MDPs is that nonnegative transition
\textit{rates}, which may not be transition probabilities, are
considered. One of the applications of such models is to the control
of branching processes; see e.g., Rothblum and
Veinott~\cite{rothblumVeinott1992} and Pliska~\cite{pliska1978}.
Other references for branching processes and other models with
transition rates greater than one are given in
Section~\ref{sec:applications}. \textit{Absorbing} MDPs, which were
introduced by Hordijk~\cite{hordijk1974} and studied in the
constrained setting by Altman~\cite{altman1999} and Feinberg and
Rothblum~\cite{feinbergRothblum2012}, can also be viewed as transient
MDPs.

It is well-known that discounted MDPs can be reduced to absorbing or
transient MDPs (see e.g.,\ \cite[p.\ 137]{altman1999}).
Theorem~\ref{thm:mainTrans} in this paper provides conditions under
which the converse is also true. In particular, the reduction comes
from a version of the similarity transformation considered by
Veinott~\cite{veinott1969}, which is attributed there to Alan
Hoffman. This reduction relates the value function and optimality
equation of the original transient model with those of the
corresponding discounted model. It implies the existence of stationary
optimal policies for transient models if certain natural conditions
hold. It also implies that the sets of optimal actions for these two
models coincide.  In the case of finite state and action sets, the
reduction shows that complexity estimates for Howard's policy
iteration algorithm for discounted MDPs imply corresponding estimates
for transient MDPs. Ye~\cite{ye2011} proved that Howard's policy
iteration algorithm, which corresponds to a block-pivoting simplex
method, and the simplex method with Dantzig's rule compute optimal
policies for discounted MDPs with a fixed discount factor in strongly
polynomial time. The complexity estimates from \cite{ye2011} were
improved in Hansen et al.~\cite{h-m-z2013} and further improved in
Scherrer~\cite{scherrer2016}. Ye~\cite{ye2011} and
Denardo~\cite{denardo2016} also obtained complexity estimates for
transient MDPs. In Section~\ref{sec:finite-state-action-transient},
Denardo's~\cite{denardo2016} estimate for Howard's policy iteration
algorithm, which corresponds to a block-pivoting simplex method, is
derived from Scherrer's~\cite{scherrer2016} estimate by using the
reduction of a transient MDP to a discounted
one.

On the other hand, the discounted-cost criterion plays an
important role in the theory of average-cost MDPs. Many results have
been proved using the so-called ``vanishing discount factor''
approach, where discounted total costs with discount factor tending to
one are used to obtain a stationary average-cost optimal policy via an
optimality inequality or equation; see e.g.,\
Sennott~\cite[Chapter~7]{sennott1999}, Sch\"{a}l~\cite{schal1993},
Hern\'{a}ndez-Lerma and Lasserre~\cite[Chapter~5]{hl-l1996}, and
Feinberg et al.~\cite{fkz2012}.

A direct reduction of average-cost MDPs to discounted ones, which
yields sufficient conditions for the existence of stationary
average-cost optimal policies, was established by Ross~\cite{ross1968,
  ross1968_1} for MDPs with Borel state space, finite action sets,
bounded costs, and a state to which the process will transition from
any state under any action with probability at least $\alpha > 0$.
This reduction and Ye's~\cite{ye2011} results were used by Feinberg and
Huang~\cite{feinbergHuang2013} to obtain iteration bounds for
average-cost policy iterations. Gubenko and
\v{S}tatland~\cite{gubenkoStatland1975} showed that a reduction is
also possible for MDPs with Borel state space, bounded costs, and
compact action sets, if a ``minorization'' condition, which generalizes
Ross's~\cite{ross1968_1} assumption, is satisfied; see also Dynkin and
Yushkevich~\cite[Chapter 7, \S 10]{dynkinYushkevich1979}.

More recently, Akian and Gaubert~\cite{akianGaubert2013} used methods
from non-linear Perron-Frobenius theory to reduce a
perfect-information zero-sum stochastic game with finite state and
action sets, containing a state being recurrent under every pair of
stationary strategies, to a discounted game with state-dependent
discount factors. In this paper, we provide a slightly modified
version of their transformation for the case of MDPs with possibly
infinite state and action spaces. This reformulation makes the
connection between their transformation and the work of
Ross~\cite{ross1968, ross1968_1} and Veinott and
Hoffman~\cite{veinott1969} more apparent. In the context of MDPs with
transition probabilities, this transformation yields a reduction of a
finite state and action average-cost problem with a state recurrent
under every stationary policy to a discounted MDP. The transformation
also allows one to write the optimality equation, prove the existence
of stationary optimal policies, and, in the case of finite state and
action sets, formulate an alternative linear program for such
average-cost problems. This program is based on the linear program
formulation for the discounted MDP, to which the original problem is
reduced. Therefore, an average-cost problem can be solved in strongly
polynomial time with complexity estimates similar to those in
Scherrer~\cite{scherrer2016}. In addition, Howard's policy iterations
for the obtained discounted MDPs coincide with Howard's policy
iterations for the initial average-cost unichain MDP. Therefore,
Scherrer's~\cite{scherrer2016} results on discounted MDPs imply that
Howard's policy iteration algorithm for the average-cost problem
computes an optimal policy in strongly polynomial time with the
complexity estimates similar to the estimates in \cite{scherrer2016}.
This also implies that, if there exists a state recurrent under all
stationary policies, the block-pivoting simplex method for the linear
programming problem for average-cost MDPs, is also strongly polynomial
with the same complexity estimates.

Previously, Zadorojniy et al.~\cite{zadorojniyEtal2009} showed that,
if every state is recurrent under every stationary policy and an MDP
satisfies a coupling property introduced there, then both discounted
and average-cost optimal policies can be computed in strongly
polynomial time. This is proved in \cite{zadorojniyEtal2009} by
introducing an algorithm that, as was shown by Even and
Zadorojniy~\cite{evenZadorojniy2012}, is equivalent to applying the
Gass-Saaty pivoting rule to the appropriate LP formulation for an
MDP. As is shown in \cite{zadorojniyEtal2009}, the aforementioned
coupling property holds for discrete-time versions of M/M/1 queues.

The model and the optimality criteria considered in ths paper are
described in Section~\ref{sec:model-description}. In
Section~\ref{sec:total-cost-mdps} we formulate the
\textit{Hoffman-Veinott (HV)} transformation \cite{veinott1969}, and
give conditions under which it leads to the reduction of the original
transient total-cost MDP to a discounted MDP with transition
probabilities. Finally, in Section~\ref{sec:average-cost-mdps} we
consider a version of Akian and Gaubert's~\cite{akianGaubert2013}
transformation for average-cost MDPs and the associated reduction to
discounted MDPs.  Most of the paper deals with countable-state MDPs.
Sections \ref{sec:finite-state-action-transient} and
\ref{sec:finite-state-action} deal with finite-state problems, while
Sections \ref{sec:extens-unco-state} and \ref{sec:extens-unco-state-1}
study MDPs with Borel state spaces.

\section{Model description}
\label{sec:model-description}

Consider a discrete-time MDP with \textit{state space} $\X$ and
\textit{action space} $\A$. Most of this paper, except
Sections~\ref{sec:extens-unco-state} and
\ref{sec:extens-unco-state-1}, deals with countable-state MDPs. We
start by introducing a countable-state MDP. Let $\X$ be countable and
$\A$ be a Borel subset of a complete separable metric space. For each
$x \in \X$, the \textit{set of available actions} $A(x)$ is a nonempty
Borel subset of $\A$. The \textit{one-step cost} function $c(x, a)$ is
(Borel-)measurable in $a \in A(x)$ for each $x \in \X$. The
\textit{transition rates} $q(y | x, a) \geq 0$ are measurable in
$a \in A(x)$ for each $x, y \in \X$ and satisfy
\begin{equation}
  \label{eq:2}
  \sup\{\textstyle\sum_{y \in \X}q(y | x, a) : x \in \X, \ a \in
  A(x)\} < \infty.
\end{equation}

\subsection{Remarks on transition rates whose sum may be greater than one}
\label{sec:applications}
The case where $\sum_{y \in \X}q(y | x, a)$ is possibly greater than
one for some state-action pairs has been studied under various
names. In Rothblum and Veinott~\cite{rothblumVeinott1992} and in
Rothblum and Whittle~\cite{rothblumWhittle1982}, such models are
called \emph{branching Markov decision chains}. They have also been
referred to as \emph{Markov population decision chains} in
\cite{veinott2008}, \cite{eavesVeinott2014}. As is explained in
Remark~\ref{rem:equivForm} below, such models can be viewed as Markov
decision processes with transition probabilities and a
state-action-dependent discount factor that is possibly greater than
one. The case of a constant discount factor, which is possibly greater
than one is studied in Hinderer and Waldmann
\cite{hindererWaldmann2005}.

Such models are applicable in a diverse array of contexts.  For
example, Markov decision models with transition rates with values
possibly greater than one appear in multi-armed bandit problems with
risk-seeking utility functions; see Denardo et
al.~\cite{denardoEtal2013, denardoEtal2007}. In addition, their
relevance to the control of multitype branching processes, which can
be used to model problems in infinite particle systems, marketing, and
population genetics, is explained in Pliska \cite{pliska1976,
  pliska1978}. Other relevant application areas are described in Eaves
and Veinott \cite{eavesVeinott2014}.

\begin{rem}\label{rem:equivForm}
  Equivalently to considering transition rates $q(\cdot | x, a)$, one
  can consider transition probabilities $p(\cdot | x, a)$ and a
  discount function $\alpha: \X \times \A \rightarrow [0, \infty)$.
  In particular, given an MDP in the latter form, let
  $q(\cdot | x, a) := \alpha(x, a)p(\cdot | x, a)$ for $x \in \X$,
  $a \in A(x)$; conversely, given transition rates $q(\cdot | x, a)$,
  let $\alpha(x, a) := q(\X | x, a)$ and
  $p(\cdot | x, a) := q(\cdot | x, a) / q(\X | x, a)$ for $x \in \X$,
  $a \in A(x)$. Expected total costs for arbitrary policies can be
  defined in a standard way via the Ionescu Tulcea Theorem (see e.g.,
  \cite[pp.\ 140-141]{bertsekasShreve1996}) by interpreting
  $\alpha(x,a)$ as a state-action dependent discount factor, and $p$
  as a transition probability. The existing literature on total-cost
  MDPs with transition rates having values possibly greater than one
  deals only with Markov policies; see e.g., \cite{pliska1976,
    pliska1978, rothblumWhittle1982, eavesVeinott2014}. This remark
  overcomes this limitation.  However, for transient total-cost models
  this remark and the reduction to a discounted MDP with transition
  probabilities and a discount factor less than one
  (Section~\ref{sec:results}) imply the optimality of stationary
  policies over all randomized history-dependent ones.  Therefore, we
  mostly consider only stationary policies in this paper. We remark
  that, when (\ref{eq:2}) holds, it is also possible to transform the 
  original total-cost problem to a discounted one with a constant
  discount factor possibly greater than one; see
  \cite[Remark~5]{hindererWaldmann2003}.
\end{rem}

\subsection{Optimality criteria}
\label{sec:optimality-criteria}

A \textit{stationary policy} is a mapping $\phi: \X \rightarrow \A$
satisfying $\phi(x) \in A(x)$ for each $x \in \X$; let $\F$ denote the
set of all such policies. It can be shown that it suffices to consider
such policies for the optimality criteria considered in this paper;
see Remarks~\ref{rem:total-nonstationary-policies} and
\ref{rem:avg-nonstationary-policies}. Under $\phi \in \F$, the
decision-maker always selects the action $\phi(x)$ when the current
state is $x$. For $\phi \in \F$, consider the matrix of one-step
transition rates $Q_\phi$ with elements $q(y | x, \phi(x))$,
$x, y \in \X$. Also, given a \textit{weight function}
$W:\X \rightarrow [1, \infty)$ and a matrix $B$ with elements
$B(x, y)$, $x, y \in \X$, let
\begin{displaymath}
\|B\|_W := \sup_{x \in \X}W(x)^{-1}\sum_{y \in \X}|B(x, y)|W(y).
\end{displaymath}
If $W(x)=1$ for all $x\in\X,$ then $||B||_W = ||B||:= \sup_{x\in\X}
\sum_{y\in\X} |B(x,y)|.$  If the function $W$ is bounded from above
and below by a finite constant $C,$ then
\begin{equation}\label{eqEFEF1}
||B||_W\le C||B||.
\end{equation}
 In particular, if $\X$ is a finite set, then \eqref{eqEFEF1} holds
 with $C=\max_{x\in\X} W(x).$

For undiscounted total costs, which are considered in
Section~\ref{sec:total-cost-mdps}, the following generalization of the
transience condition
studied in Veinott~\cite{veinott1969} and Pliska~\cite{pliska1978} is
assumed to hold.
\begin{asmpT} \
  \begin{enumerate}[(i)]
  \item   The MDP is \textnormal{transient}, that is,\ there is a weight function
  $V:\X \rightarrow [1, \infty)$ and a constant $K \geq 1$
  that satisfy
  \begin{equation}
    \label{eq:transienceAsmp}
    \| \sum_{n = 0}^\infty Q_\phi^n \|_V \leq K < \infty \quad \text{for
      all} \ \phi \in \F.
  \end{equation}
\item   There is a constant $\bar{c} < \infty$ satisfying
  \begin{equation}\label{eq:c-Vbounded}
    \sup_{a \in A(x)}|c(x, a)| \leq \bar{c}V(x) \qquad \text{for all}
    \ x \in \X.
  \end{equation}
\item For every $x \in \X$ the mapping
  \begin{displaymath}
    a \mapsto \sum_{y \in \X}q(y | x, a)V(y) < \infty, \qquad a \in A(x),
  \end{displaymath}
  is continuous on $A(x)$.
  \end{enumerate}

\end{asmpT}

For $V \equiv 1$, a number of conditions sufficient for or equivalent
to (\ref{eq:transienceAsmp}) are provided in
Pliska~\cite{pliska1978}. If the state and action sets are finite,
then Assumption T is equivalent to the assumption that there exists a
constant $K$ such that
$\|\sum_{n = 0}^\infty Q_\phi^n\| \leq K < \infty$. For finite state
and action sets, Assumption~T can be checked in strongly polynomial
time using the procedure described in \cite[proof of
Theorem~1]{veinott1974}, where it is attributed to Eric Denardo; see
also \cite[Lemma~10]{denardo2016}.

For $\phi \in \F$, let $c_\phi(x) := c(x, \phi(x))$ for $x \in
\X$. Under Assumption~T, the \textit{total cost} incurred under
$\phi \in \F,$ when the initial state is $x \in \X,$ is
\begin{displaymath}
  v^\phi(x) := \sum_{n = 0}^\infty Q_\phi^n c_\phi(x).
\end{displaymath}
A policy $\phi_*$ is \textit{total-cost optimal} if
$v^{\phi_*}(x) = \inf_{\phi \in \F}v^\phi(x) =: v(x)$ for all $x \in \X$.

 The following characterization of Assumption~T will be used to define the
transformations described in Sections~\ref{sec:hoffm-vein-transf} and
\ref{sec:hv-ag-transformation} for total-cost MDPs.
\begin{prp}\label{prp:SH}
  Assumption~T(i) holds if and only if there is a  function
  $\mu: \X \rightarrow [1,\infty)$ such that $V(x) \leq \mu(x) \leq
  KV(x)$ for all $x \in \X$ and
  \begin{equation}\label{eq:SH}
    \mu(x) \geq V(x) + \sum_{y \in \X}q(y | x,
      a)\mu(y), \quad x \in \X, \ a \in A(x).
  \end{equation}
\end{prp}
\begin{proof}
  Suppose there is a function $\mu:\X \rightarrow [1, \infty)$ that satisfies
  $V(x) \leq \mu(x) \leq KV(x)$ for all $x \in \X$ and (\ref{eq:SH}).
  Consider an arbitrary $\phi \in \F$. According to (\ref{eq:SH}),
  \begin{displaymath}
    \mu(x) \geq V(x) + \sum_{y \in \X}q(y | x, \phi(x))\mu(y) \qquad
    \text{for all} \ \ x \in \X,
  \end{displaymath}
  which, since $\mu$ is nonnegative and majorized by $KV$, implies
  that for $N = 1, 2, \dots$
  \begin{displaymath}
    KV(x) \geq \sum_{n=0}^{N-1}Q_\phi^nV(x) + Q_\phi^N\mu(x) \geq
    \sum_{n=0}^{N-1}Q_\phi^nV(x) \qquad \text{for all} \ \ x \in \X.
  \end{displaymath}
  Hence
  \begin{equation}
    \label{eq:16}
  K \geq
  V(x)^{-1}\lim_{N\rightarrow\infty}\sum_{n=0}^{N-1}Q_\phi^nV(x)
  \qquad \text{for all} \ \ x \in \X.
  \end{equation}
  Since $\phi \in \F$ is arbitrary, it follows from (\ref{eq:16})
  that Assumption~T holds.

  Conversely, suppose Assumption~T holds and consider the operator $\U$ defined for functions $u: \X
  \rightarrow [0, \infty)$ by
  \begin{displaymath}
    \U u(x) := \sup_{A(x)}\left[V(x) + \sum_{y \in \X}q(y | x,
      a)u(y)\right], \quad x \in \X.
  \end{displaymath}
  Let $u_0 := V$, and for $n = 1, 2, \dots$ let
  $u_n = \U u_{n - 1}$. Note that the positivity of $V$ implies $V \leq
  u_n \leq u_{n+1}$ for all $n$. Furthermore, letting $\mu :=
  \lim_{n\rightarrow\infty}u_n$, Lebesgue's monotone convergence
  theorem implies that $\mu = \U\mu$. Hence to complete a proof, it
  suffices to show that $u_n \leq KV$ for all $n$.

  Note that $u_0 = V \leq KV$ because $K \geq
  1$. Next, suppose $u_n \leq KV$ for some nonnegative integer $n$,
  and consider an arbitrary $\epsilon > 0$. Let $\phi^\epsilon$ be a
  stationary policy satisfying
  \begin{displaymath}
    V(x) + \sum_{y \in \X}q(y | x, \phi^\epsilon(x))u_n(y) \geq
    \U u_n(x) - \epsilon(KV(x))^{-1}, \qquad x \in \X.
  \end{displaymath}
  Define $\tilde{u}_0 := u_n.$ For $N = 1, 2, \dots$ let
  \begin{equation}\label{eq:tilde-u}
    \tilde{u}_N(x) := \sum_{i=0}^{N-1}Q_{\phi^\epsilon}^iV(x) +
    Q_{\phi^\epsilon}^Nu_n(x), \qquad x \in \X.
  \end{equation}
  Since $0 \leq u_n \leq KV$, it follows that $0 \leq Q_{\phi^\epsilon}^Nu_n \leq
  KQ_{\phi^\epsilon}^NV$ for all $N,$ which according to Assumption~T
  implies that $Q_{\phi^\epsilon}^N u_n \rightarrow 0$ as $N
  \rightarrow \infty$. Hence it follows from (\ref{eq:tilde-u}) and
  Assumption~T that
  \begin{equation}
    \label{eq:17}
    \lim_{N\rightarrow\infty}\tilde{u}_N(x) \leq \sum_{i=0}^\infty
    Q_{\phi^\epsilon}^iV(x) \leq KV(x) \qquad \text{for all} \ \ x \in \X.
  \end{equation}
  Next, we claim that
  \begin{equation}
    \label{eq:15}
    \tilde{u}_N(x) \geq u_{n+1}(x) -
    \epsilon(KV(x))^{-1}\sum_{i=0}^{N-1}Q_{\phi^\epsilon}^iV(x) \qquad
    \text{for all} \ \ x \in \X, \ N \geq 1.
  \end{equation}
  Observe that (\ref{eq:17}) and (\ref{eq:15}) together with
  Assumption~T imply
  \begin{displaymath}
    KV(x) \geq u_{n+1}(x) - \epsilon \qquad \text{for all} \ \ x \in \X.
  \end{displaymath}
  Since $\epsilon > 0$ is arbitrary, this implies by induction that
  $u_n \leq KV$ for all $n$, from which the Proposition follows. To
  verify that (\ref{eq:15}) holds, first observe that for all
  $x \in \X$,
  \begin{displaymath}
    \tilde{u}_1(x) = V(x) + Q_{\phi^\epsilon}u_n(x) \geq \U u_n(x) -
    \epsilon(KV(x))^{-1} = u_{n+1}(x) - \epsilon(KV(x))^{-1}.
  \end{displaymath}
  Next, suppose $\tilde{u}_N \geq u_{n+1} -
  \epsilon(KV)^{-1}\sum_{i=0}^{N-1}Q_{\phi^\epsilon}^iV$ for some $N
  \geq 1$. Then, since $u_{n+1} \geq u_n$, it follows that for $x \in \X$
  \begin{align*}
    \tilde{u}_{N+1}(x) &= V(x) + Q_{\phi^\epsilon}\tilde{u}_N(x) \\
    &\geq V(x) + Q_{\phi^\epsilon}u_{n+1}(x) -
      \epsilon(KV(x))^{-1}\sum_{i=0}^{N-1}Q_{\phi^\epsilon}^{i+1}V(x) \\
    &\geq V(x) + Q_{\phi^\epsilon}u_n(x) -
      \epsilon(KV(x))^{-1}\sum_{i=1}^NQ_{\phi^\epsilon}^i V(x) \\
    &\geq \U u_n(x) - \epsilon(KV(x))^{-1} -
      \epsilon(KV(x))^{-1}\sum_{i=1}^NQ_{\phi^\epsilon}^i V(x) \\
    &= u_{n+1}(x) - \epsilon(KV(x))^{-1}\sum_{i=0}^{(N+1) -
      1}Q_{\phi^\epsilon}^iV(x).
  \end{align*}
\end{proof}

\begin{lem}\label{lem:mu-integral-cont}
  Suppose statements (i) and (iii) of Assumption~T hold, and let
  $\mu$ be the function described in the statement of
  Proposition~\ref{prp:SH}. Further, suppose that for every
  $x,y \in \X$ the mappings $a \mapsto q(y | x, a)$ and
  $a \mapsto q(\X | x, a)$ are continuous on $A(x)$. Then for every
  $x \in \X$ the mapping
  \begin{displaymath}
    a \mapsto \sum_{y \in \X}q(y | x, a)\mu(y), \qquad a \in A(x),
  \end{displaymath}
  is continuous on $A(x)$.
\end{lem}
\begin{proof}
  Fix $x \in \X$, and let $\{a_n\}$ be any sequence in $A(x)$ converging
  to some $a \in A(x)$. Under the hypotheses of the lemma, the
  sequence of measures $\{q(\cdot | x, a_n)\}$ converges setwise to
  $q(\cdot | x, a)$; for a definition of setwise convergence of
  measures, see e.g., \cite[p.\ 269]{royden1988}. Since
  $0 \leq \mu(x) \leq KV(x)$ for all $x \in \X$, and
  $\sum_{y \in \X}q(y | x, a)V(y) < \infty$, it follows from the
  dominated convergence theorem for setwise converging measures (see
  e.g., \cite[Proposition~18]{royden1988}) that
  \begin{displaymath}
    \lim_{n\rightarrow\infty}\sum_{y \in \X}q(y | x, a_n)\mu(y) =
    \sum_{y \in \X}q(y | x, a)\mu(y).
  \end{displaymath}
\end{proof}

\begin{rem}\label{rem:tau}
  For $\phi \in \F$ and $x \in \X$, let
  $\tau^\phi(x) := \sum_{n = 0}^\infty Q_\phi^ne(x)$ and
  $\tau(x) := \sup_{\phi \in \F}\tau^\phi(x)$. Then it follows from
  \cite[Proposition~9.6.4]{hl-l1999} that $\tau(x) \leq KV(x)$ for all
  $x \in \X$. When the transition rates $q$ are substochastic, that
  is, $\sum_{y \in \X}q(y | x, a) \leq 1$ for all $x \in \X$ and
  $a \in A(x)$ (if equality holds for all $x$ and $a$, then $q$ is
  called \textit{stochastic}), the quantity $\tau^\phi(x)$ can be
  interpreted as the expected total lifetime of the process under the
  policy $\phi$ when $x$ is the initial state.
\end{rem}

For average costs, which are dealt with in
Section~\ref{sec:average-cost-mdps}, Assumption~HT on hitting times formulated below is assumed to
hold. To state it, for $z \in \X$ and $\phi \in \F$ consider the matrix
$\mathbin{_z Q_\phi}$ with elements
\begin{displaymath}
  \mathbin{_z Q_\phi}(x, y) :=
  \begin{cases}
    q(y | x, \phi(x)), &\quad \text{if} \ x \in \X, \ y \neq z, \\
    0, &\quad \text{if} \ x \in \X, \ y = z.
  \end{cases}
\end{displaymath}
\begin{asmpHT} \
  \begin{enumerate}[(i)]
  \item   There is a state $\ell \in \X$ and a constant $K^*$ satisfying
  \begin{equation}\label{eq:HT}
    \| \sum_{n = 0}^\infty \mathbin{_\ell Q_\phi^n}\| \leq K^* < \infty
    \quad \text{for all} \ \phi \in \F.
  \end{equation}
\item The one-step cost function $c$ is bounded.
  \end{enumerate}
\end{asmpHT}

\begin{rem} \label{rem:T-vs-HT} Observe that Assumptions T and HT are
  related.  If an MDP satisfies Assumption HT then, if state $\ell$
  and all transition rates to it are removed, the truncated MDP is
  transient with $V \equiv 1$.  In particular, when the transition
  rates are substochastic or the sets $\X$ and $A(\ell)$ are finite,
  Assumption HT for the initial MDP and Assumption T with $V \equiv 1$
  for the MDP with the state $\ell$ removed are equivalent.  For the
  substochastic case, $K^*\le K+1,$ where $K$ is the constant from
  Assumption T for the truncated MDP. This is true because the truncated
  MDP does not contain the state $\ell$, whereas $K^*$ is an upper bound
  on the mean recurrence time for all the states of the original MDP, including the state $\ell,$ under any stationary
  policy. 
\end{rem}

When $q$ is substochastic, Assumption~HT means that when the initial
state is $x$, the expected hitting time to state $\ell$ under any
stationary policy is bounded above by $K^*$. When the state and action
sets are finite, Assumption~HT is equivalent to state $\ell$ being
recurrent under all stationary policies. According to Feinberg and
Yang~\cite{feinbergYang2008}, Assumption~HT can be checked in strongly
polynomial time. We remark that any MDP satisfying Assumption~HT is
unichain, and that in general the problem of checking whether an MDP
is unichain is NP-hard~\cite{tsitsiklis2007}.  In addition,
Assumption~HT is related to many other recurrence conditions that have
been used to study average-cost MDPs; see e.g.,\ the surveys by
Federgruen et al.~\cite{federgruenHordijkTijms1978},
Thomas~\cite{thomas1980}, and Hern\'{a}ndez-Lerma et
al.~\cite{hl-mdo-cc1991}.

For the initial state $x \in \X$, the \textit{average cost}
incurred under $\phi \in \F$ is
\begin{displaymath}
  w^\phi(x) := \limsup_{N \rightarrow \infty}\frac{1}{N}\sum_{n = 0}^{N - 1}Q_\phi^nc_\phi(x).
\end{displaymath}
A policy $\phi_*$ is \textit{average-cost optimal} if $w^{\phi_*}(x) =
\inf_{\phi \in \F}w^\phi(x) =: w(x)$ for all $x \in \X$.

According to Rothblum~\cite{rothblum1975}, a stationary policy $\phi$
is called \emph{normalized} if $\sum_{n=0}^\infty \beta^nQ_\phi^n$
converges for all $\beta \in (0,1).$ If Assumption T holds or the
transition rates $q$ are substochastic, then any stationary policy is
normalized.  Given $\beta \in [0, 1)$ and an initial state $x \in \X$,
the \textit{$\beta$-discounted cost} incurred under a normalized
policy $\phi \in \F$ is
\begin{displaymath}
  v_\beta^\phi(x) := \sum_{n = 0}^\infty \beta^n Q_\phi^nc_\phi(x).
\end{displaymath}
A policy $\phi_*$ is \textit{$\beta$-optimal} if
$v_\beta^{\phi_*}(x) = \inf_{\phi \in \F}v_\beta^\phi(x) =:
v_\beta(x)$ for all $x \in \X$.

In this paper, transformations to discounted MDPs with stochastic
transition rates are considered.  Discounted MDPs with nonstochastic
transition rates are mentioned only in Remark~\ref{rem3}, where
complexity estimates for discounted MDPs with transition rates
satisfying Assumption T are provided.

%


\section{Undiscounted total costs}
\label{sec:total-cost-mdps}

The transformation of the original transient MDP to a discounted one,
which we call the \textit{Hoffman-Veinott (HV)} transformation, is
given in Section~\ref{sec:hoffm-vein-transf}. Under the hypotheses of
Theorem~\ref{thm:mainTrans} in Section~\ref{sec:results}, a stationary
optimal policy exists for the transformed discounted MDP, and the sets
of optimal policies for the transformed and original MDPs
coincide. The finite state and action case is considered in
Section~\ref{sec:finite-state-action-transient}. The Borel-state case
is treated in Section~\ref{sec:extens-unco-state}.

\subsection{HV transformation}
\label{sec:hoffm-vein-transf}

Let Assumption~T hold. By Proposition~\ref{prp:SH}, there is a
nonnegative function $\mu$ on $\X$ that satisfies $V \leq \mu \leq KV$
and (\ref{eq:SH}). Objects associated with the discounted MDP will be
indicated by a tilde. The state space is
$\tilde{\X} := \X \cup \{\tilde{x}\}$, where $\tilde{x} \not\in \X$ is
a cost-free absorbing state. Letting $\tilde{a}$ denote the only
action available at state $\tilde{x}$, the action space is
$\tilde{\A} := \A \cup \{\tilde{a}\}$ and for $x \in \tilde{\X}$ the
set of available actions is unchanged if $x \in \X$, namely
\begin{displaymath}
  \tilde{A}(x) :=
  \begin{cases}
    A(x), &\quad \text{if} \ x \in \X,\\
    \{\tilde{a}\}, &\quad \text{if} \ x = \tilde{x}.
  \end{cases}
\end{displaymath}
Define the one-step costs $\tilde{c}$ as
\begin{displaymath}
  \tilde{c}(x, a) :=
  \begin{cases}
    \mu(x)^{-1}c(x, a), &\quad \text{if} \ x \in \X, \ a \in A(x),\\
    0, &\quad \text{if} \ (x, a) = (\tilde{x}, \tilde{a}).
  \end{cases}
\end{displaymath}
To complete the definition of the discounted MDP, choose a discount
factor
\begin{displaymath}
  \tilde{\beta} \in \left[\frac{K - 1}{K}, 1\right),
\end{displaymath}
and let
\begin{equation}\label{eq:pTilde}
  \tilde{p}(y | x, a) :=
  \begin{cases}
    \frac{1}{\tilde{\beta}\mu(x)}q(y | x, a)\mu(y), &\quad \text{if} \
    x,y \in \X, \ a \in A(x), \\
    1 - \frac{1}{\tilde{\beta}\mu(x)}\sum_{y \in \X}q(y | x, a)\mu(y),
    &\quad \text{if} \ y = \tilde{x}, \ x \in \X, \ a \in A(x), \\
    1 &\quad \text{if} \ y = x = \tilde{x}, \ a = \tilde{a}.
  \end{cases}
\end{equation}
Note that $\tilde{p}(\cdot | x, a)$ is a probability distribution on
$\tilde{\X}$ for each $x \in \tilde{\X}$ and $a \in \tilde{A}(x)$.
Also, since $\tilde{A}(\tilde{x})$ is a singleton, the sets of
policies for these two models coincide. Let
$\tilde{v}_{\tilde{\beta}}^\phi(x)$ denote the
$\tilde{\beta}$-discounted cost incurred under the policy $\phi$ when
the initial state of this MDP is $x \in \tilde{\X}$, and let
$\tilde{v}_{\tilde{\beta}}(x) = \inf_{\phi \in
  \F}\tilde{v}_{\tilde{\beta}}^\phi(x)$ for $x \in \tilde{\X}$.

\paragraph{Relation to Veinott's positive similarity transformation.}

Veinott's~\cite{veinott1969} positive similarity transformation is
defined for transient MDPs with finite state and action sets as
follows. Given a diagonal matrix $B$ with positive diagonal entries,
let
\begin{displaymath}
  \tilde{c}_\phi := Bc_\phi \quad \text{and} \quad \tilde{P}_\phi :=
  B Q_\phi B^{-1}, \qquad \phi \in \F.
\end{displaymath}
According to Veinott~\cite{veinott1969}, properties that are invariant
under this transformation include the transience of a policy, the
optimality of a policy, and the geometric convergence of value
iteration to the unique fixed point of the optimality
operator. Further, letting $\mu$ be the unique vector satisfying
\begin{equation}\label{eq:hoffmanIdea}
  \mu(x) = \max_{\phi \in \F}[1 + Q_\phi \mu(x)], \quad x \in \X,
\end{equation}
and letting $\mu(x)^{-1}$ be the nonzero entry on the $x$-th row of
$B$, it follows from \cite[Lemma~3]{veinott1969} that if the spectral
radii of the matrices $Q_\phi$ are all less than one, then the row
sums of the matrices $\tilde{P}_\phi$ are all less than one; Veinott
attributes this result to Alan Hoffman. The first line of
(\ref{eq:pTilde}) is an implementation of Veinott's similarity
transformation that is applicable to all policies. Transformations of
the form $\mu(x)^{-1}q(y | x, a)\mu(y)$ have also been used in the
literature to reduce MDPs with unbounded one-step costs to MDPs with
bounded one-step costs; see e.g.,\ \cite[p.\ 101]{wal1981}.


\subsection{Results}
\label{sec:results}

Given $\phi \in \F$, the following proposition relates the total costs
incurred in the original undiscounted MDP with those incurred in the
discounted MDP defined by the HV transformation.
\begin{prp}\label{prp:transPolEval}
  Suppose statements (i) and (ii) of Assumption~T hold. 
  Then the one-step cost function $\tilde c$ is bounded and $v^\phi(x) = \mu(x)\tilde{v}_{\tilde{\beta}}^\phi(x)$ for each
  $\phi \in \F$ and $x \in \X$. 
\end{prp}
\begin{proof}
  Consider the matrix $\tilde{P}_\phi$ with elements
  $\tilde{P}_\phi(x, y) := \tilde{p}(y | x, \phi(x))$, $x, y \in
  \X$. Then
  \begin{equation}
    \label{eq:10}
    \tilde{v}_{\tilde{\beta}}^\phi(x) = \sum_{n = 0}^\infty
    \tilde{\beta}^n\tilde{P}_\phi^n \tilde{c}_\phi(x), \quad x \in \tilde{\X}.
  \end{equation}
  Since the state $\tilde{x}$ is cost-free and absorbing, it follows
  from the definitions of $\tilde{P}_\phi$ and $\tilde{c}_\phi$ that
  \begin{equation}
    \label{eq:11}
    \tilde{\beta}^n\tilde{P}_\phi^n\tilde{c}_\phi(x) =
    \mu(x)^{-1}Q_\phi^n c_\phi(x) \quad \text{for all} \ x \in \X, \ n
    \geq 0.
  \end{equation}
  Observe that, since $\mu$ majorizes $V$, according to
  (\ref{eq:c-Vbounded}) the mapping $x \mapsto \mu(x)^{-1}c_\phi(x)$
  is bounded. Hence, combining (\ref{eq:10}) and (\ref{eq:11}), for
  $x \in \X$
  \begin{displaymath}
    \tilde{v}_{\tilde{\beta}}^\phi(x) = \mu(x)^{-1}\sum_{n = 0}^\infty
    Q_\phi^nc_\phi(x) = \mu(x)^{-1}v^\phi(x).
  \end{displaymath}
  Proposition~\ref{prp:SH} and Assumption T(ii) imply that $|\tilde{c}(x,a)|\le {\bar c}$ for all $x\in\X$ and $a\in A(x).$ 
\end{proof}

The optimality results in this section and Section~\ref{sec:results-1}
rely on the following compactness-continuity conditions.
\begin{infComp}[{cf. \cite[p.\ 181]{feinberg2002}}] \
  \begin{enumerate}[(i)]
  \item $A(x)$ is compact for each $x \in \X$;
  \item $c(x, a)$ is lower semicontinuous in $a \in A(x)$ for each $x \in
    \X$;
  \item the transition rates $q(y | x, a)$ are continuous in
    $a \in A(x)$ for each $x, y \in \X$;
  \item the transition rates $q(\X | x, a) := \sum_{y \in \X}q(y | x,
    a)$ are continuous in $a \in A(x)$ for each $x \in \X$.
  \end{enumerate}
\end{infComp}
\noindent Observe that, if the state set is finite, then assumption
(iii) of the Compactness Conditions implies assumption (iv). Also, if
the transition rates are stochastic, that is,\ $q(\X | x, a) = 1$ for
all $x \in \X$ and $a \in A(x)$, then assumption (iv) of the
Compactness Conditions always holds.

\begin{lem}\label{lem:compactnessHV}
  Suppose Assumption~T and the Compactness Conditions hold. Then the
  discounted MDP defined by the HV transformation also satisfies the
  Compactness Conditions.
\end{lem}
\begin{proof}
  Assumptions (i)-(ii) of the Compactness Conditions imply that the sets
  $\tilde{A}(x)$ are compact and $\tilde{c}$ is bounded and is lower
  semicontinuous in $a$. In addition, assumption (iii) of the Compactness
  Conditions and Lemma~\ref{lem:mu-integral-cont} imply that $\tilde{p}(y | x, a)$ is continuous in $a
  \in A(x)$ for all $x, y \in \X$, and assumption (iv) implies that
  $\tilde{p}(\tilde{x} | x, a)$ is continuous in $a \in A(x)$ for all
  $x \in \X$. Since $\tilde{p}$ is also stochastic, it follows that
  the Compactness Conditions hold for the transformed MDP.
\end{proof}

The main result (Theorem~\ref{thm:mainTrans}) of this section relies
on the following proposition. To state it, for $\beta \in [0, 1)$
define
\begin{equation}
  \label{eq:13}
  A_\beta^*(x) := \left\{a \in A(x) \ \left\vert\vphantom{\frac{1}{1}}\right. \ v_\beta(x) = c(x, a) +
    \beta \sum_{y \in \X}q(y | x, a)v_\beta(y)\right\}, \quad x \in \X.
\end{equation}
\begin{prp}[cf.\ {\cite[pp.\ 181,
    184]{feinberg2002}}]\label{prp:total1}
  If an MDP with transition probabilities $q$ and bounded one-step
  costs $c$ satisfies the
  Compactness Conditions, then for any discount factor $\beta \in [0,
  1)$:
  \begin{enumerate}[(i)]
  \item the value function $v_\beta$ is the unique bounded function
    satisfying the optimality equation
  \begin{equation}\label{eq:HVOE}
    v_\beta(x) = \min_{A(x)}\left[c(x, a) +
      \beta\sum_{y \in \X}q(y | x,
      a)v_\beta(y)\right], \quad x \in \X;
  \end{equation}
  \item there is a stationary $\beta$-optimal policy;
  \item a policy $\phi \in \F$ is $\beta$-optimal if and only if
    $\phi(x) \in A_{\beta}^*(x)$ for all $x \in \X$.
  \end{enumerate}
\end{prp}
\begin{proof}
  The Compactness Conditions imply that, if $\X$ is endowed with the
  discrete topology, then the transition probabilities $q$ are weakly
  continuous in $(x, a)$ where $x \in \X$ and $a \in A(x)$. This
  implies that the MDP satisfies Assumption~(W$^*$) in
  \cite{fkz2012}. The validity of (\ref{eq:HVOE}) and statements (ii),
  (iii) follows from \cite[Theorem~2]{fkz2012}. The uniqueness claim
  in (i) follows from the contraction mapping principle; see
  Denardo~\cite{denardo1967} for details.
\end{proof}

To state Theorem~\ref{thm:mainTrans}, let
\begin{equation}\label{eq:A-star}
  A^*(x) := \left\{ a \in A(x) \ \left\vert\vphantom{\frac{1}{1}}\right. \ v(x) = c(x, a) +
    \sum_{y \in \X}q(y | x, a) v(y)\right\}, \quad x \in \X,
\end{equation}
where $v$ is the value function of the original undiscounted total
cost MDP. 
\begin{thm}\label{thm:mainTrans}
  Suppose the original undiscounted total-cost MDP satisfies
  Assumption~T and the Compactness Conditions. Then:
  \begin{enumerate}[(i)]
  \item the value function $v = \mu \tilde{v}_{\tilde{\beta}}$ is the
    unique function   satisfying the optimality
    equation
  \begin{equation}\label{eq:5-countable}
    v(x) = \min_{A(x)}\left[c(x, a) +
      \sum_{y \in \X}q(y | x,
      a)v(y)\right], \quad x \in \X, 
  \end{equation}
   and such that
     \begin{equation}\label{eq:v-Vbounded-countable}
       \|v\|_V := \sup_{x \in \X}V(x)^{-1}|v(x)| < \infty;
    \end{equation}
  \item there is a stationary total-cost optimal policy;
  \item a policy $\phi \in \F$ is total-cost optimal if and only if
    $\phi(x) \in A^*(x)$ for all $x \in \X$, and
    \begin{equation}
      \label{eq:12-countable}
      A^*(x) = \left\{ a \in A(x) \ \left\vert\vphantom{\frac{1}{1}}\right. \ \tilde{v}_{\tilde{\beta}}(x) = \tilde{c}(x, a) +
        \tilde{\beta}\sum_{y \in \tilde{\X}}\tilde{p}(y | x, a) \tilde{v}_{\tilde{\beta}}(y)\right\}, \quad x \in \X;
    \end{equation}
    in other words, the sets of optimal actions for the original
    transient MDP and for the transformed discounted MDP with
    transition probabilities $\tilde{p}$ coincide.
  \end{enumerate}
\end{thm}
\begin{proof}
  By Lemma~\ref{lem:compactnessHV}, the transformed discounted MDP
  satisfies the Compactness Conditions. Hence statements (i)-(iii) of
  Proposition~\ref{prp:total1} hold for the transformed MDP.

  
  Straightforward calculations show that the function $v = \mu \tilde{v}_{\tilde{\beta}}$ satisfies the optimality equation \eqref{eq:5-countable} if and only if the function
  $v_\beta:=\tilde{v}_{\tilde{\beta}}$ satisfies the optimality equation \eqref{eq:v-Vbounded-countable} for the
  $\tilde{\beta}$-discounted MDP defined by the HV
  transformation.  In view of Proposition~\ref{prp:SH}, $\|v\|_V<\infty$ if and only if the function $\tilde{v}_{\tilde{\beta}}$ is bounded.  Lemma~\ref{lem:compactnessHV} and Propositions~\ref{prp:transPolEval},~\ref{prp:total1} imply statement (i). 

  According to Proposition~\ref{prp:total1}(i), there is a $\phi_* \in
  \F$ that is $\tilde{\beta}$-optimal for the transformed MDP. By
  Proposition~\ref{prp:transPolEval}, $v^{\phi_*} = \mu
  \tilde{v}_{\tilde{\beta}}^{\phi_*} = \mu \tilde{v}_{\tilde{\beta}} =
  v$, so $\phi_*$ is total-cost optimal for the original
  MDP. Therefore (ii) holds.

  It follows from the definitions of $\tilde{\X}$, $\tilde{A}$,
  $\tilde{c}$, $\tilde{\beta}$, and $\tilde{p}$ that (\ref{eq:12-countable})
  holds. Suppose $\phi \in \F$ is total-cost optimal for the original
  MDP. Then $v^\phi = v$, so since $v^\phi = c_\phi + Q_\phi v^\phi$
  it follows that $\phi(x) \in A^*(x)$ for all $x \in \X$. Conversely,
  if $\phi(x) \in A^*(x)$ for all $x \in \X$, then according to
  Proposition~\ref{prp:total1}(iii) and (\ref{eq:12}) the policy
  $\phi$ is $\tilde{\beta}$-optimal for the transformed MDP. By
  Proposition~\ref{prp:transPolEval}, this means $\phi$ is total-cost
  optimal for the original MDP. Hence (iii) holds.
\end{proof}

\begin{cor}\label{cor:mainTrans}
  Suppose Assumption~T and the Compactness Conditions hold. If an
  algorithm computes an optimal policy for the discounted
  MDP defined by the HV transformation, then this policy is optimal
  for the original undiscounted total-cost MDP.
\end{cor}

\begin{rem}\label{rem:total-nonstationary-policies}
  The HV transformation also applies to arbitrary policies if the
  total costs are defined using the equivalent formulation in terms of
  transition probabilities and state-dependent discount factors; see
  Remark~\ref{rem:equivForm}. Since stationary policies are optimal
  within the class of all policies for discounted MDPs with transition
  probabilities satisfying the Compactness Conditions \cite[p.\
  184]{feinberg2002}, the stationary total-cost optimal policies
  referred to in Theorem~\ref{thm:mainTrans} are optimal over
  nonstationary policies as well.

\end{rem}

\subsection{Finite state and action sets}
\label{sec:finite-state-action-transient}

In this section, we assume that both $\X$ and $\A$ are finite. Recall
from the paragraph after the statement of Assumption~T that, when the
state and action sets are finite, Assumption~T is equivalent to the
existence of a constant $K$ such that
\begin{equation}\label{eq:T-finite}
\sum_{n = 0}^\infty Q_\phi^n e(x) \leq K \qquad \text{for all}  \ \phi
\in \F, \ x
\in \X,
\end{equation}
where $e$ denotes the function on $\X$ that is identically equal to
one. Therefore, in this section we assume without loss of
generality that (\ref{eq:T-finite}) holds.

Corollary~\ref{cor:mainTrans} implies that an optimal
policy for the original transient MDP can be computed by solving the
following linear program (LP): 
\begin{equation}\label{eq:discLPtrans}
  \begin{aligned}
    \text{minimize} &\quad \sum_{x \in \tilde{\X}}\sum_{a \in
      \tilde{A}(x)}\tilde{c}(x, a) \tilde{z}_{x,a} &\\
    \text{such that} &\quad \sum_{a \in \tilde{A}(x)}\tilde{z}_{x, a} -
    \tilde{\beta}\sum_{y \in \tilde{\X}}\sum_{a \in \tilde{A}(y)}\tilde{p}(x |
    y, a)\tilde{z}_{y, a} = 1, & x \in \tilde{\X}, \\
    &\quad \tilde{z}_{x, a} \geq 0, & x \in \tilde{\X}, \ a \in \tilde{A}(x).
  \end{aligned}
\end{equation}
According to Scherrer~\cite[Theorem~3]{scherrer2016}, the LP
(\ref{eq:discLPtrans}) can be solved using
\begin{equation}
  \label{eq:32}
  (m-n)\left\lceil \frac{1}{1 - \tilde{\beta}} \log \frac{1}{1 -
      \tilde{\beta}}\right\rceil = O((m-n) K\log K)
\end{equation}
iterations of the block-pivoting simplex method corresponding to
Howard's policy iteration algorithm. Alternatively, if the simplex
method with Dantzig's rule is applied to the LP
(\ref{eq:discLPtrans}), then according to
\cite[Theorem~4]{scherrer2016} at most
\begin{equation}
  \label{eq:33}
  n(m-n)\left(1 + \frac{2}{1 - \tilde{\beta}} \log \frac{1}{1 -
      \tilde{\beta}}\right) = O(n(m-n)K \log K)
\end{equation}
iterations are needed to compute an optimal
solution.

Let $z_{x, a} := \tilde{z}_{x, a}/\mu(x)$ for $x \in \X$ and
$a \in A(x).$
The  LP~(\ref{eq:discLPtrans})
for this discounted MDP can be written as
\begin{equation}\label{eq:discLPtrans-1}
  \begin{aligned}
    \text{minimize} &\quad \sum_{x \in \X}\sum_{a \in
      A(x)}c(x, a) z_{x, a} &\\
    \text{such that} &\quad \sum_{a \in A(x)}z_{x, a} -
    \sum_{y \in \X}\sum_{a \in A(y)}q(x |
    y, a)z_{y, a} = \mu(x)^{-1}, & x \in \X, \\
    &\quad z_{x, a} \geq 0, & x \in \X, \ a
    \in A(x).
  \end{aligned}
\end{equation}
This is true because of the following arguments that hold all $x\in\X$ and for all $a\in A(x):$ 
(i) the objective functions for the LPs~(\ref{eq:discLPtrans}) and (\ref{eq:discLPtrans-1}) are equal
because $\tilde A(x)=A(x),$ $c(x,a)=\mu(x){\tilde c}(x,a), $
$\tilde A({\tilde x})=\{{\tilde x}\},$ and $ {\tilde c}({\tilde x},{\tilde a})=0,$ (ii)  for $x,$ the equality constraints are equivalent in these LPs because ${\tilde p}(x|{\tilde x},{\tilde a})=0,$
$q(y|x,a)={\tilde\beta}{\tilde p}(y|x,a),$ where $y\in \X,$  and the inequality constraints are equivalent because  $\mu(x)>0,$
(iii) the equality and inequality constraints for ${\tilde x}$ can be
excluded from the LP~(\ref{eq:discLPtrans}) because the former implies
that ${\tilde z}_{{\tilde x},{\tilde a}}=(1-{\tilde\beta})^{-1}(
1+\tilde{\beta}\sum_{y\in\X}\sum_{a\in A(y)}{\tilde p}({\tilde x}|y,
a)\tilde{z}_{y, a})>0$ and, in view of (i) and (ii), the variable ${\tilde z}_{{\tilde x},{\tilde a}}$ does not appear anywhere else in the LP~(\ref{eq:discLPtrans}).

Since a policy for this discounted MDP is optimal if and only if it is
optimal for the original discounted MDP defined by the HV
transformation, Corollary~\ref{cor:mainTrans} implies that an optimal
policy for the original transient MDP can be computed by solving the
LP (\ref{eq:discLPtrans-1}). Since any optimal policy derived from the
LP (\ref{eq:discLPtrans-1}) is still optimal if the right-hand sides
of the equality constraints are replaced by arbitrary positive numbers
(see e.g., \cite[Corollary~3, Remark~6]{hordijkKallenberg1984}), it
follows that an optimal policy for the original transient MDP can be
computed by solving the LP
\begin{equation}\label{eq:discLPtrans-2}
  \begin{aligned}
    \text{minimize} &\quad \sum_{x \in \X}\sum_{a \in
      A(x)}c(x, a) z_{x, a} &\\
    \text{such that} &\quad \sum_{a \in A(x)}z_{x, a} -
    \sum_{y \in \X}\sum_{a \in A(y)}q(x |
    y, a)z_{y, a} = 1, & x \in \X, \\
    &\quad z_{x, a} \geq 0, & x \in \X, \ a
    \in A(x).
  \end{aligned}
\end{equation}
This provides an alternative derivation of the LP~\eqref{eq:discLPtrans-2} for transient MDPs
provided in Denardo~\cite{denardo2016}, where it is shown that the
LP~\eqref{eq:discLPtrans-2} can be solved using at most $(m-n)k^*$
iterations of the block-pivoting simplex method corresponding to
Howard's policy iteration algorithm, where $k^*$ is the smallest
integer $k$ that satisfies $1 > K(1 - (1/K))^k$
\cite[Theorem~2]{denardo2016}.  This implies that the required number
of iterations is $O((m-n)K \log K)$, which matches the estimate
(\ref{eq:32}) for the LP
(\ref{eq:discLPtrans}) obtained using \cite[Theorem~3]{scherrer2016}.
If $K = 1$, then
$\tilde{\beta} = 0$ and the problem can be solved by simply selecting,
for each $x \in \X$, an action minimizing $c(x, a)$ over $a \in A(x)$.
Denardo~\cite{denardo2016} also showed that the LP~\eqref{eq:discLPtrans-2} 
can be solved using at most $(m - n)j^*$ iterations of the simplex
method with Dantzig's rule, where $j^*$ is the smallest integer $j$
that satisfies $1 > (e\tau)(1 - (1/e\tau))^j$ and $\tau$ is the
function defined in Remark~\ref{rem:tau}. This implies that the
simplex method with Dantzig's rule requires at most
\begin{equation}\label{eq:36}
O((m-n)(e\tau) \log(e\tau))
\end{equation}
iterations to solve the LP (\ref{eq:discLPtrans-2}).

\begin{rem}
  Applying the simplex method with Dantzig's rule to
  (\ref{eq:discLPtrans}) can be viewed as applying a certain pivoting
  rule to the LP (\ref{eq:discLPtrans-2}). In particular, given a
  non-optimal basic feasible solution to (\ref{eq:discLPtrans-2})
  corresponding to the non-optimal stationary policy $\phi$, the
  variable $z_{x, a}$ that enters the basis under this pivoting rule
  is the one minimizing
  \begin{equation}
    \label{eq:34}
    \tilde{c}(x, a) + \tilde{\beta}\sum_{y \in \X}\tilde{p}(y | x,
    a)\tilde{v}_{\tilde{\beta}}^\phi(y)  - v_{\tilde{\beta}}^\phi(x) = \frac{1}{\mu(x)}\left[c(x,
      a) + \sum_{y \in \X}q(y | x, a)v^\phi(y) - v^\phi(x)\right],
  \end{equation}
  where the expression in the square brackets on the right-hand side
  of (\ref{eq:34}) is precisely the reduced cost, for the variable
  $z_{x, a}$, associated with the basis corresponding to $\phi$. It
  follows from (\ref{eq:33}) that this pivoting rule for the LP
  (\ref{eq:discLPtrans-2}) considered in Denardo~\cite{denardo2016} is
  strongly polynomial when $K$ is fixed. This algorithm is not the
  same as applying Dantzig's rule to the LP (\ref{eq:discLPtrans-2}),
  however; see Remark~\ref{rem:dantzig-counterexample-total} below.
\end{rem}

\begin{rem}\label{rem:dantzig}
  To compare the estimates (\ref{eq:33}) and (\ref{eq:36}) for the simplex method with
  Dantzig's rule for LPs \eqref{eq:discLPtrans} and \eqref{eq:discLPtrans-2} respectively, consider
  the functions  $f(n, K) := nK \log K$ and
  $g(e\tau) := e\tau \log(e\tau)$. Using these notations, the estimate
  (\ref{eq:33}) is
  \begin{equation}\label{eq:37}
    O((m-n)f(n, K)),
  \end{equation}
  and the estimate (\ref{eq:36}) is
  \begin{equation}\label{eq:38}
    O((m-n)g(e\tau)).
  \end{equation}
 If $K$ is fixed, then the estimate (\ref{eq:37}) is better
  than the estimate (\ref{eq:38}).  This is because, when $K$ fixed,
  $f(n, K) = O(n)$ while $nK \geq e\tau \geq n-1 + K$ implies that
  $g(e\tau) = O(n \log n)$. In addition, if $\tau \equiv K$, then the
  estimate (\ref{eq:37}) is also better than (\ref{eq:38}) because
  $g(e\tau) = nK\log nK \geq nK\log K = f(n, K)$. On the other hand,
  for some particular values of $n$, $K$, and $\tau$ it is possible that
  $f(n, K) > g(e\tau)$. For example, consider the MDP with $n=10$
  states and 1 action per state, where for states 1 through 9 the
  process stops after one transition, and for state 10 the process
  stops with probability 1/5 and continues with probability 4/5. Then
  $K=5$ and $e\tau = 9 + 5 = 14$, which implies
  $f(n, K) \geq 10\cdot 5 \cdot \log(5) > 14 \cdot \log(14) =
  g(e\tau)$.
\end{rem}

\begin{rem}\label{rem:transientHoward}
  Consider Howard's policy iteration algorithm for the discounted MDP
  defined by the HV transformation, which according to \cite[p.\
  68]{kallenberg1983} is equivalent to a block-pivoting simplex method
  for the LP (\ref{eq:discLPtrans}). Given $\phi \in \F$ and recalling
  that $\tilde{x}$ is a cost-free absorbing state, an improved policy
  $\phi^+ \in \F$ is constructed (when possible) as follows. For each
  $x \in \X$, $\phi^+(x)$ is taken to be any action belonging to
  \begin{equation}
    \label{eq:26}
    \argmin_{a \in A(x)}\left[\tilde{c}(x, a) + \tilde{\beta}\sum_{y
        \in \X}\tilde{p}(y | x, a)\tilde{v}_{\tilde{\beta}}^\phi(y)\right].
  \end{equation}
  It follows from the definitions of $\tilde{c}$ and $\tilde{p}$ and
  Proposition~\ref{prp:transPolEval} that for each $x \in \X$, the set
  (\ref{eq:26}) is equal to
  \begin{equation}
    \label{eq:27}
    \argmin_{a \in A(x)}\left[ c(x, a) + \sum_{y \in \X}q(y | x, a)v^\phi(y) \right].
  \end{equation}
  Under Howard's policy iteration algorithm for the original transient
  MDP, which according to the arguments in \cite[p.\
  68]{kallenberg1983} and \cite[pp. 55-56]{hordijkKallenberg1984} is
  equivalent to a block-pivoting simplex method for the LP
  (\ref{eq:discLPtrans-2}), given $\phi \in \F$ an improved policy
  $\phi^+$ is constructed (when possible) by taking, for each
  $x \in \X$, $\phi^+(x)$ to be any action belonging to
  (\ref{eq:27}). Since for each $x \in \X$ the sets (\ref{eq:26}) and
  (\ref{eq:27}) are equal, it follows that there is a one-to-one
  correspondence between sequences of policies generated by Howard's
  policy iteration algorithm for the discounted MDP defined by the HV
  transformation, and sequences of policies generated by Howard's
  policy iteration algorithm for the original transient MDP. Using
  Scherrer's~\cite[Theorem~3]{scherrer2016} $O(mK \log K)$ iteration
  bound for Howard's policy iteration algorithm for discounted MDPs,
  we therefore obtain the bound derived by
  Denardo~\cite[Theorem~2]{denardo2016} for the original transient
  MDP.

\end{rem}

\begin{rem}\label{rem:dantzig-counterexample-total}
  According to Remark~\ref{rem:transientHoward}, starting with the
  same basic variables, the sequences of basic variables for
  implementations of block-pivoting simplex methods for the LPs
  (\ref{eq:discLPtrans}) and (\ref{eq:discLPtrans-2}) coincide. This
  is not true for the simplex method with Dantzig's rule, however. To
  confirm this, consider the following transient MDP. The set of
  states is $\X = \{1,2\},$ 
  and the sets of available actions are
  $A(1) = A(2) = \{a, b\}$. The transition rates are
  $q(1 | 1, a) = 2/3$, $q(2 | 1, a) = 1/6$,
  $q(1 | 1, b) = q(2 | 1, b) = 1/3$, $q(1 | 2, a) = 2/3$,
  $q(2 | 2, a) = 1/6$, $q(1 | 2, b) = 1/12$, and $q(2 | 2, b) =
  5/6$. The one-step costs are $c(1, a) = -0.91$, $c(1, b) = -0.56$,
  $c(2, a) = -0.19$, and $c(2, b) = -0.8$. One can verify that the
  function $\mu$ defined by $\mu(1) = 8$ and $\mu(2) = 10$ satisfies
  (\ref{eq:SH}) with $V \equiv 1$. The total-cost LP given by
  (\ref{eq:discLPtrans-2}) is
  \begin{equation}\label{eq:LP-total}
  \begin{aligned}
    \text{minimize} &\quad -0.91z_{1, a} - 0.56z_{1, b} - 0.19z_{2, a} -0.8z_{2,b}
    \\
    \text{such that} &\quad \frac{1}{3}z_{1, a} + \frac{2}{3}z_{1, b} -
    \frac{2}{3}z_{2, a} - \frac{1}{12}z_{2, b} = 1\\
    &\quad -
    \frac{1}{6}z_{1, a} - \frac{1}{3}z_{1, b}  + \frac{5}{6}z_{2, a} + \frac{1}{6}z_{2, b} = 1\\
    &\quad z_{1, a}, z_{1, b}, z_{2, a}, z_{2, b} \geq 0,
  \end{aligned}
\end{equation}
and, letting $K = 10$, the LP (\ref{eq:discLPtrans}) for the
discounted MDP defined by the HV transformation with discount factor
$\tilde{\beta} = (K-1)/K = 9/10$ is
  \begin{equation}\label{eq:LP-discounted}
  \begin{aligned}
    \text{minimize} &\quad -0.11375\tilde{z}_{1, a} - 0.07\tilde{z}_{1, b} - 0.019\tilde{z}_{2, a} -0.08\tilde{z}_{2,b}
    \\
    \text{such that} &\quad \frac{1}{3}\tilde{z}_{1, a} + \frac{2}{3}\tilde{z}_{1, b} -
    \frac{8}{15}\tilde{z}_{2, a} - \frac{1}{15}\tilde{z}_{2, b} = 1\\
    &\quad -
    \frac{5}{24}\tilde{z}_{1, a} - \frac{5}{12}\tilde{z}_{1, b}  + \frac{5}{6}\tilde{z}_{2, a} + \frac{1}{6}\tilde{z}_{2, b} = 1\\
    &\quad \tilde{z}_{1, a}, \tilde{z}_{1, b}, \tilde{z}_{2, a}, \tilde{z}_{2, b} \geq 0,
  \end{aligned}
\end{equation}
where, according to the remarks following (\ref{eq:discLPtrans-1}),
the variable $\tilde{z}_{\tilde{x}, \tilde{a}}$ has been removed. For
both LPs, suppose the initial basic feasible solution for the simplex method with Dantzig's rule is the one defined by the stationary policy $\phi$ 
with $\phi(1) = b$ and $\phi(2) = a$; namely, for both LPs the basic variables are those corresponding to the state-action pairs $(1, b)$ and $(2, a)$. Consider the first iterations of the simplex method with Dantzig's rule for the LPs~(\ref{eq:LP-total}) and (\ref{eq:LP-discounted}. 
For the LP (\ref{eq:LP-total}), the basic variable $z_{2, a}$ is the unique variable to leave the basis, while for the LP (\ref{eq:LP-discounted}) the basic variable $\tilde{z}_{1, b}$ is the unique variable to leave the basis.
\end{rem}

\begin{rem}\label{rem3}
  If (\ref{eq:T-finite}) holds, it holds with the same upper bound $K$
  if the transition rates $q$ are replaced with the transition rates
  $\beta q$, where $\beta\in (0,1].$ Hence the results in
  Denardo~\cite[Theorems~1, 2]{denardo2016} and the estimates
  (\ref{eq:32}), (\ref{eq:33}) above imply that the number of
  arithmetic operations needed to compute an optimal policy for a
  discounted MDP satisfying Assumption T can be bounded by a
  polynomial in $m$ that does not depend on the discount factor
  $\beta\in (0,1].$ In particular, these bounds hold for all discount
  factors $\beta\in (0,1].$ If $\beta=0,$ the discounted problem
  becomes a one-step problem, which is equivalent to a problem with
  $K=1;$ this case was discussed in the paragraph preceding
  Remark~\ref{rem:dantzig}.
\end{rem}

\begin{rem}\label{rem:computing-K}
  For $x \in \X$, let
  $\tau(x) := \sup_{\phi \in \F}\sum_{n=0}^\infty Q_\phi^n e(x)$. Then
  $K_\tau := \max_{x \in \X}\tau(x)$ is the smallest constant $K$
  satisfying (\ref{eq:T-finite}). The natural question is how to
  compute $K_\tau$. One method to compute $K_\tau$ consists in the
  following. First, compute an optimal policy $\phi_*$ for a transient
  MDP that is identical to the original MDP except that all one-step
  costs are equal to $-1$. Then, compute the value function
  $v^{\phi_*}$ of this optimal policy, and set
  $K_\tau = \max_{x \in \X}v^{\phi_*}(x)$. As discussed in the
  paragraph following (\ref{eq:discLPtrans-2}), the policy $\phi_*$
  can be computed using $O((m-n)K_\tau \log K_\tau)$ iterations of
  Howard's policy iteration algorithm. Further, the function
  $v^{\phi_*}$ can be computed by solving a system of $n$ linear
  equations using Gaussian elimination in $O(n^3)$ arithmetic
  operations; for other methods see e.g., \cite{strassen1969}, \cite{c-w1990}. 
\end{rem}

\subsection{Extension to uncountable state spaces}
\label{sec:extens-unco-state}

In this section, we assume that the state space $\X$ is a Borel subset
of a complete separable metric space, and that the transition rates are
defined by a Borel-measurable transition kernel $q$ on $\X$ given
$\Gr(A) := \{(x, a) : x \in \X, \ a \in A(x)\}$, which we assume to be
a Borel subset of $\X \times \A$. That  is, 
$q(\cdot | x, a)$ is a finite measure for
every $(x, a) \in \Gr(A)$, and $q(B | \cdot)$ is a Borel-measurable
function on $\Gr(A)$ for every Borel subset $B$ of $\X$. In addition,
the one-step cost function $c:\Gr(A) \rightarrow \R$ is
Borel-measurable.

The set of stationary policies $\F$ is identified with the set of all
Borel-measurable functions $\phi:\X \rightarrow \A$ satisfying
$\phi(x) \in A(x)$ for all $x \in \X$.  To formulate a version of
Assumption~T in this setting, for $\phi \in \F$ define the operator
$Q_\phi$ for Borel-measurable functions $u:\X \rightarrow \R$ by
\begin{equation}\label{eq:q-Borel}
  Q_\phi u(x) := \int_\X u(y)q(dy | x, \phi(x)), \qquad x \in \X,
\end{equation}
and given a Borel-measurable weight function $W:\X \rightarrow \R$ and
a Borel-measurable transition kernel $B(\cdot | \cdot)$ on $\X$ given $\X$,
let
\begin{displaymath}
  \|B\|_W := \sup_{x \in \X}W(x)^{-1}\int_\X W(y)B(dy | x).
\end{displaymath}

\begin{asmpTprime} \
  \begin{enumerate}[(i)]
  \item   There is a Borel-measurable weight function
  $V:\X \rightarrow [1, \infty)$ and a constant $K \geq 1$
  that satisfy
  \begin{equation}
    \label{eq:transienceAsmp-1}
    \| \sum_{n = 0}^\infty Q_\phi^n \|_V \leq K < \infty \quad \text{for
      all} \ \phi \in \F.
  \end{equation}
\item   Moreover, there is a constant $\bar{c} < \infty$ satisfying
  \begin{displaymath}
    \sup_{a \in A(x)}|c(x, a)| \leq \bar{c}V(x) \qquad \text{for all}
    \ x \in \X,
  \end{displaymath}
  and for every $x \in \X$ the mapping
  \begin{displaymath}
    a \mapsto \int_{y \in \X}V(y)q(y | x, a) < \infty, \qquad a \in A(x),
  \end{displaymath}
  is continuous on $A(x)$.
  \end{enumerate}
\end{asmpTprime}

To obtain a reduction to a discounted MDP, we consider the following
setwise-continuity and compactness conditions:

\begin{samepage}
  \begin{asmpS} \
    \begin{enumerate}[(a)]
    \item Statements (i) and (ii) of the Compactness Conditions hold.
    \item For every $x \in \X$, 
    if the
      sequence $\{a_n\}$ in $A(x)$ converges to $a \in A(x)$, then for
      every Borel subset $B$ of $\X$ the sequence $\{q(B | x, a_n)\}$
      converges to $q(B | x, a)$.
    \end{enumerate}
  \end{asmpS}
\end{samepage}

\begin{prp}\label{prp:SH-1}
  Suppose Assumption~S holds. Then Assumption~T'(i) holds if and only if
  there is a Borel-measurable function $\mu:\X \rightarrow [1,
  \infty)$ satisfying $V(x) \leq \mu(x) \leq KV(x)$ and
  \begin{displaymath}
    \mu(x) \geq V(x) + \int_\X \mu(y)q(dy | x, a), \qquad (x, a) \in \Gr(A).
  \end{displaymath}
\end{prp}
\begin{proof}
  This follows from the proof of Proposition~\ref{prp:SH}, with all
  sums replaced with integrals, and by applying the Brown and
  Purves~\cite[Corollary~1]{brownPurves1973} theorem on
  Borel-measurable selection.
\end{proof}
In this setting, the analogue of Lemma~\ref{lem:mu-integral-cont}
holds as well.
\begin{lem}\label{lem:mu-integral-cont-1}
  Suppose Assumption~S and statements (i) and (iii) of T' hold, and
  let $\mu$ be the Borel-measurable function described in the
  statement of Proposition~\ref{prp:SH-1}. Then for every $x \in \X$
  the mapping
  \begin{displaymath}
    a \mapsto \int_\X \mu(y)q(dy | x, a), \qquad a \in A(x),
  \end{displaymath}
  is continuous on $A(x)$.
\end{lem}
\begin{proof}
  This follows from the proof of Lemma~\ref{lem:mu-integral-cont},
  where all sums are replaced with integrals.
\end{proof}

\subsubsection{HV transformation}
\label{sec:hv-transformation}

Let $\B(\X)$ denote the Borel $\sigma$-algebra of $\X$. The definition
of the HV transformation in the setting of a possibly uncountable
state space is identical to the definition presented in
Section~\ref{sec:hoffm-vein-transf}, except that the cost-free
absorbing state $\tilde{x}$ is taken to be isolated from the original
state space $\X$, and the transition probability kernel $\tilde{p}$ is
defined by
\begin{displaymath}
  \tilde{p}(B | x, a) :=
  \begin{cases}
    \frac{1}{\tilde{\beta}\mu(x)}\int_B\mu(y)q(dy | x, a), &\quad
    \text{if} \ B \in \B(\X), \ (x, a) \in \Gr(A), \\
    1 - \frac{1}{\tilde{\beta}\mu(x)}\int_\X \mu(y)q(dy | x, a),
    &\quad \text{if} \ B = \{\tilde{x}\}, \ (x, a) \in \Gr(A), \\
    1, &\quad \text{if} \ B = \{\tilde{x}\}, \ (x, a) = (\tilde{x}, \tilde{a}).
  \end{cases}
\end{displaymath}

\subsubsection{Results}
\label{sec:results-2}

\begin{prp}
  Suppose Assumptions~S and T' hold. Then $v^\phi(x) =
  \mu(x)\tilde{v}_{\tilde{\beta}}^\phi(x)$ for each $\phi \in \F$ and
  $x \in \X$.
\end{prp}
\begin{proof}
  This follows from the proof of Proposition~\ref{prp:transPolEval} by
  defining for $\phi \in \F$ the operator $\tilde{P}_\phi$ applied to integrable
  Borel-measurable functions $u:\X \rightarrow \R$,
  \begin{displaymath}
    \tilde{P}_\phi u(x) := \int_{\tilde{\X}}u(y)\tilde{p}(dy | x,
    \phi(x)), \qquad x \in \X.
  \end{displaymath}
\end{proof}

\begin{lem}
  Suppose Assumptions~S and T' hold. Then the discounted MDP defined
  by the HV transformation also satisfies Assumption~S.
\end{lem}
\begin{proof}
  This follows from the proof of Lemma~\ref{lem:compactnessHV},
  Lemma~\ref{lem:mu-integral-cont-1}, and the fact that the added
  cost-free absorbing state $\tilde{x}$ is isolated from $\X$.
\end{proof}

The special case of Theorem~\ref{thm:transient-Borel} below for
$V \equiv 1$ was proved by Pliska~\cite[Theorem~1.3]{pliska1978}. To
state Theorem~\ref{thm:transient-Borel}, for $\beta \in (0, 1)$ and
$x \in \X,$ define the sets $A_\beta^*(x)$ and $A^*(x)$ by replacing
the sums in (\ref{eq:13}) and (\ref{eq:A-star}), respectively, with
integrals.
\begin{samepage}
\begin{thm}\label{thm:transient-Borel}
  Suppose the original undiscounted total-cost MDP satisfies
  Assumptions~S and T'. Then:
  \begin{enumerate}[(i)]
  \item the value function $v = \mu \tilde{v}_{\tilde{\beta}}$ is the
    unique Borel-measurable function   satisfying the optimality
    equation
  \begin{equation*}
    v(x) = \min_{A(x)}\left[c(x, a) +
      \int_\X v(y)q(dy | x,
      a)\right], \quad x \in \X, 
  \end{equation*}
   and such that
     \begin{equation*}
       \sup_{x \in \X}V(x)^{-1}|v(x)| < \infty;
     \end{equation*}
  \item there is a stationary total-cost optimal policy;
  \item a policy $\phi \in \F$ is total-cost optimal if and only if
    $\phi(x) \in A^*(x)$ for all $x \in \X$, and
    \begin{equation}
      \label{eq:12}
      A^*(x) = \left\{ a \in A(x) \ \left\vert\vphantom{\frac{1}{1}}\right. \ \tilde{v}_{\tilde{\beta}}(x) = \tilde{c}(x, a) +
        \tilde{\beta}\int_\X \tilde{v}_{\tilde{\beta}}(y) \tilde{p}(dy | x, a)\right\}, \quad x \in \X;
    \end{equation}
    in other words, the sets of optimal actions for the original
    transient MDP and for the transformed discounted MDP with
    transition probabilities $\tilde{p}$ coincide.
  \end{enumerate}
\end{thm}
\end{samepage}
\begin{proof}
  This follows from the proof of Theorem~\ref{thm:mainTrans}, where
  instead of \cite{fkz2012} one can use \cite[Proposition~2.1]{schal1993}.
\end{proof}

\section{Average costs per unit time}
\label{sec:average-cost-mdps}

In Section~\ref{sec:hv-ag-transformation}, we provide a slight
modification of the transformation introduced by Akian and
Gaubert~\cite{akianGaubert2013}. Since it can be viewed as an
extension of the HV transformation described in
Section~\ref{sec:hoffm-vein-transf}, we refer to the transformation
given in Section~\ref{sec:hv-ag-transformation} as the \textit{HV-AG}
transformation. Like the HV transformation, the HV-AG transformation
produces a discounted MDP with transition probabilites. According to
Theorem~\ref{thm:optEqs} in Section~\ref{sec:results-1}, for an
average-cost MDP with transition probabilities $q$ satisfying
Assumption~HT and the Compactness Conditions given in
Section~\ref{sec:results}, the HV-AG transformation reduces the
original problem to a discounted one. The finite state and action case
is considered in Section~\ref{sec:finite-state-action}. The
Borel-state case is treated in Section~\ref{sec:extens-unco-state-1}.

\subsection{HV-AG transformation}
\label{sec:hv-ag-transformation}

Suppose Assumption~HT holds. According to Proposition~\ref{prp:SH},
there is a function $\mu: \X \rightarrow [1, \infty)$ that
satisfies $\mu \leq K^*$ and
\begin{equation}
  \label{eq:1}
  \mu(x) \geq 1 + \sum_{y \in \X \setminus \{\ell\}}q(y | x, a)\mu(y),
  \quad x \in \X , \ a \in A(x).
\end{equation}

Objects associated with the discounted MDP will be indicated by a
horizontal bar. The state space is $\bar{\X} := \X \cup \{\bar{x}\}$,
where $\bar{x} \not\in \X$ is a cost-free absorbing state. Letting
$\bar{a}$ denote the only action available at state $\bar{x}$, the
action space is $\bar{\A} := \A \cup \{\bar{a}\}$ and for
$x \in \bar{\X}$ the set of available actions is unchanged if
$x \in \X$, namely
\begin{displaymath}
  \bar{A}(x) :=
  \begin{cases}
    A(x), &\quad \text{if} \ x \in \X, \\
    \{\bar{a}\}, &\quad \text{if} \ x = \bar{x}.
  \end{cases}
\end{displaymath}
Define the one-step
costs $\bar{c}$ by
\begin{displaymath}
  \bar{c}(x, a) :=
  \begin{cases}
    \mu(x)^{-1}c(x, a), &\quad \text{if} \ x \in \X, \ a \in A(x), \\
    0, &\quad \text{if} \ (x, a) = (\bar{x}, \bar{a}).
  \end{cases}
\end{displaymath}
To complete the definition of the discounted MDP, choose a discount factor
\begin{displaymath}
  \bar{\beta} \in \left[\frac{K^*-1}{K^* }, 1\right),
\end{displaymath}
 and let
\begin{displaymath}
  \bar{p}(y | x, a) :=
  \begin{cases}
    \frac{1}{\bar{\beta} \mu(x)}q(y | x, a)\mu(y), & y \in
    \X \setminus \{\ell\}, \ x \in \X, \ a \in A(x),\\
    \frac{1}{\bar{\beta} \mu(x)}[\mu(x) - 1 - \sum_{y \in \X \setminus
      \{\ell\}}q(y | x, a)\mu(y)], & y = \ell, \ x \in \X, \ a \in A(x) \\
    1 - \frac{1}{\bar{\beta} \mu(x)}[\mu(x) - 1], & y =
    \bar{x}, \ x \in \X, \ a \in A(x) \\
    1, & y = \bar{x}, \ (x, a) = (\bar{x}, \bar{a}).
  \end{cases}
\end{displaymath}
Since $\mu$ satisfies (\ref{eq:SH}), $\bar{p}(\cdot | x, a)$ is a
probability distribution on $\X$ for each $x \in \bar{\X}$ and
$a \in \bar{A}(x)$. In addition, the definition of $\bar{A}$ implies
that the sets of policies for the transformed MDP and the original MDP
coincide. Let $\bar{v}_{\bar{\beta}}^\phi(x)$ denote the
$\bar{\beta}$-discounted cost incurred when the initial state of the
transformed MDP is $x \in \bar{\X}$ and the policy $\phi$ is used, and
let $\bar{v}_{\bar{\beta}}(x) := \inf_{\phi \in
  \F}\bar{v}_{\bar{\beta}}(x)$ for $x \in \bar{\X}$.
\begin{rem}
  While the HV-AG transformation applies to transition rates in
  general, the major results in Section~\ref{sec:results-1} pertain to
  the case when these rates are probabilities.
\end{rem}
\begin{rem}
  Akian and Gaubert~\cite{akianGaubert2013} prove their results by
  transforming a perfect-information mean-payoff stochastic game into
  a discounted game with state-dependent discount factors. The version
  of their transformation presented above uses techniques from
  \cite{feinberg2002_smdps} to directly obtain a problem with a single
  discount factor.
\end{rem}
\begin{rem}\label{rem:relat-rosss-transf}
  Ross~\cite{ross1968, ross1968_1} considered MDPs with transition
  probabilities $q$ satisfying the special case of Assumption~HT where
  there is a constant $\alpha$ such that
\begin{displaymath}
  q(\ell | x, a) \geq \alpha > 0 \quad \text{for all} \ \ x \in \X, \ a
  \in A(x),
\end{displaymath}
and introduced a transformation of the transition probabilities that
can be used to reduce the average-cost MDP to a discounted one. In
fact, Ross's~\cite{ross1968, ross1968_1} transformation can be viewed
as a special case of the HV-AG transformation. Namely, taking
$\mu \equiv K = 1/\alpha$, the resulting transition probabilities are
the same in both cases and the one-step costs differ by a factor of
$\alpha$.
\end{rem}

\begin{rem}
  The HV-AG transformation does not apply to the version of
  Assumption~HT with the norm $\|\cdot\|$ being replaced with
  $\|\cdot\|_V,$ when $V$ is unbounded. In particular, $\bar{p}(\bar{x} | x,
  a) \geq 0$ implies that $\mu(x) \leq (1 - \bar{\beta})^{-1}$.
\end{rem}

\subsection{Results}
\label{sec:results-1}

The proofs of Proposition~\ref{prp:1} and Theorem~\ref{thm:optEqs}
below rely on the following lemma.
\begin{lem}\label{lem:1}
  If a bounded function $f: \bar{\X} \rightarrow \R$ satisfies
  $f(\bar{x}) = 0$, then for all $x \in \X$ and $a \in A(x)$
  \begin{equation}\label{eq:mainLem}
    \bar{c}(x, a) + \bar{\beta}\sum_{y \in \bar{\X}}\bar{p}(y | x, a)f(y) =
    \frac{1}{\mu(x)}\left[c(x, a) + \sum_{y \in \X}q(y | x,
      a)\mu(y)[f(y) - f(\ell)] + [\mu(x) - 1]f(\ell)\right].
  \end{equation}
\end{lem}
\begin{proof}
  According to the definition of $\bar{c}$, $\bar{\beta}$, and
  $\bar{p}$ in Section~\ref{sec:hv-ag-transformation}, for $x \in \X$
  and $a \in A(x)$\small{
  \begin{align*}
    \bar{c}(x, a) + \bar{\beta}\sum_{y \in \bar{\X}}\bar{p}(y | x,
    a)f(y) &= \frac{c(x, a)}{\mu(x)} + \frac{1}{\mu(x)}\sum_{y \in \X \setminus
             \{\ell\}}q(y | x, a)\mu(y)f(y)
             + \frac{1}{\mu(x)}\left[\mu(x) - 1 -
             \sum_{\X \setminus \{\ell\}}q(y | x,
             a)\mu(y)\right]f(\ell) \\
           &= \frac{1}{\mu(x)}\left[ c(x, a) + \sum_{y \in \X}q(y | x,
             a)\mu(y)[f(y) - f(\ell)] + [\mu(x) - 1]f(\ell)\right].
  \end{align*}}
\end{proof}

Given $\phi \in \F$, the following proposition relates the average
costs incurred in the original MDP with the discounted costs incurred
in the MDP constructed using the HV-AG transformation. Recall that $q$
is \textit{stochastic} if $\sum_{y \in \X}q(y | x, a) = 1$ for all
$x \in \X$ and $a \in A(x)$.


\begin{prp}\label{prp:1}
  Suppose Assumption~HT holds. Let $\phi \in \F$ be a stationary policy and
  $h^\phi(x) := \mu(x)[\bar{v}_{\bar{\beta}}^\phi(x) -
  \bar{v}_{\bar{\beta}}^\phi(\ell)]$ for $x \in \X$. Then
  \begin{equation}\label{eq:correspPol}
    \bar{v}_{\bar{\beta}}^\phi(\ell) + h^\phi(x) =
    c(x, \phi(x)) + \sum_{y \in \X}q(y | x,
    \phi(x))h^\phi(y), \quad x \in \X.
  \end{equation}
  In addition, if the transition rates $q$ are stochastic,
  then $w^\phi \equiv \bar{v}_{\bar{\beta}}^\phi(\ell)$.
\end{prp}
\begin{proof}
  Since the state $\bar{x}$ in the discounted MDP defined by the HV-AG
  transformation is cost-free and absorbing, (\ref{eq:correspPol})
  follows from the fact that
  \begin{displaymath}
    \bar{v}_{\bar{\beta}}^\phi(x) = \bar{c}(x, \phi(x)) +
    \bar{\beta}\sum_{y \in \bar{\X}}\bar{p}(y | x,
    \phi(x))\bar{v}_{\bar{\beta}}^\phi(y), \quad x \in \X,
  \end{displaymath}
  and Lemma~\ref{lem:1}. Iterating (\ref{eq:correspPol}) gives
  \begin{equation}
    \label{eq:9}
    N\bar{v}_{\bar{\beta}}^\phi (\ell) + h^\phi(x) = \sum_{n = 0}^{N -
      1}Q_\phi^n c_\phi(x) + Q_\phi^Nh^\phi(x), \quad x \in \X, \ N =
    1, 2, \dots \ .
  \end{equation}
  Since $c$ is bounded, the function $h^\phi$ is bounded as well. The
  equality $w^\phi \equiv \bar{v}_{\bar{\beta}}^\phi(\ell)$ then
  follows by dividing both sides of (\ref{eq:9}) by $N$ and letting $N
  \rightarrow \infty$.

\end{proof}


\begin{lem}\label{lem:compactnessHVAG}
  Suppose Assumption~HT and the Compactness Conditions hold. Then the
  discounted MDP defined by the HV-AG transformation also satisfies
  the Compactness Conditions.
\end{lem}
\begin{proof}
  Assumptions (i)-(ii) of the Compactness Conditions imply that the
  sets $\bar{A}(x)$ are compact and $\bar{c}$ is bounded and is lower
  semicontinuous in $a$. Assumption (iii) of the Compactness
  Conditions and Lemma~\ref{lem:mu-integral-cont} imply that
  $\bar{p}(y | x, a)$ is continuous in $a \in A(x)$ for all $x \in \X$
  and $y \in \X \setminus \{\ell\}$. Assumption (iii), for state
  $\ell$, and assumption (iv) of the Compactness Conditions imply that
  $\bar{p}(\ell | x, a)$ is continuous in $a \in A(x)$ for all
  $x \in \X$.
\end{proof}

For $x \in \X$, and a constant $w$ and function $h: \X \rightarrow \R$
satisfying the average-cost optimality equation (\ref{eq:acoe}) given
in the statement of Theorem~\ref{thm:optEqs} below, consider the sets
of actions
\begin{equation}\label{eq:A-star-av}
  A_{\text{av}}^*(x) := \left\{a \in A(x) \ \left\vert\vphantom{\frac{1}{1}}\right. \ w + h(x) = c(x, a) + \sum_{y \in
      \X}q(y | x, a) h(y)\right\}, \quad x \in \X.
\end{equation}
Theorem~\ref{thm:optEqs} also follows from Federgruen and
Tijms~\cite[Theorems~2.1, 2.2]{federgruenTijms1978}, where other
recurrence conditions are considered as well.
\begin{thm}\label{thm:optEqs}
  Suppose the original MDP with transition probabilities $q$ satisfies
  Assumption~HT and the Compactness Conditions. Then:
  \begin{enumerate}[(i)]
  \item the constant $w = \bar{v}_{\bar{\beta}}(\ell)$ and the
    function
    $h(x) = \mu(x)[\bar{v}_{\bar{\beta}}(x) -
    \bar{v}_{\bar{\beta}}(\ell)]$, $x \in \X$, satisfy the optimality
    equation
  \begin{equation}\label{eq:acoe}
    w + h(x) = \min_{A(x)}\left[c(x, a) + \sum_{y \in \X}q(y | x,
      a)h(y)\right], \quad x \in \X,
  \end{equation}
  and $\bar{v}_{\bar{\beta}}(\ell)$ is the optimal average cost for
  each initial state.
\item there is a $\phi \in \F$ satisfying $\phi(x) \in A_{\text{av}}^*(x)$
  for all $x \in \X$, where
  \begin{equation}
    \label{eq:14}
    A_{\text{av}}^*(x) = \left\{ a \in A(x) \ \left\vert\vphantom{\frac{1}{1}}\right. \ \bar{v}_{\bar{\beta}}(x) = \bar{c}(x, a) +
        \bar{\beta}\sum_{y \in \bar{\X}}\bar{p}(y | x, a) \bar{v}_{\bar{\beta}}(y)\right\}, \quad x \in \X,
  \end{equation}
  and any such policy is average-cost optimal.
  \end{enumerate}
\end{thm}
\begin{proof}
  Lemma~\ref{lem:compactnessHVAG} implies that statements (i)-(iii) of
  Proposition~\ref{prp:total1} hold for the transformed MDP. In
  particular, there is a stationary $\bar{\beta}$-optimal policy $\phi$ for
  the transformed MDP, which satisfies $\phi(x) \in
  A_{\bar{\beta}}^*(x)$ for all $x \in \X$.

  The validity of (\ref{eq:acoe}) follows from applying Lemma~\ref{lem:1} to the optimality
  equation for the $\bar{\beta}$-discounted MDP defined by the HV-AG
  transformation. Further, Proposition~\ref{prp:1} implies that the
  optimal average cost for each state is
  $\bar{v}_{\bar{\beta}}(\ell)$, so (i) holds.


  Lemma~\ref{lem:1} implies that (\ref{eq:14}) holds, from which the
  existence of a $\phi \in \F$ satisfying $\phi(x) \in
  A_{\text{av}}^*(x)$ for all $x \in \X$ follows. Moreover, since the
  function $h$ is bounded, 
  \begin{displaymath}
    \lim_{N\rightarrow\infty}\frac{1}{N}\E_x^\phi h(x_N) = 0 \qquad
    \text{for all} \ x \in \X.
  \end{displaymath}
  It therefore follows from e.g., \cite[Theorem~5.2.4]{hl-l1996} that
  any $\phi \in \F$ satisfying $\phi(x) \in A_{\text{av}}^*(x)$ for
  all $x \in \X$ is average-cost optimal.
\end{proof}

\begin{cor}\label{cor:alg1}
  Suppose Assumption~HT and the Compactness Conditions hold. If
  an algorithm computes an optimal policy for the discounted MDP
  defined by the HV-AG transformation, then this policy is optimal for
  the original average-cost MDP.
\end{cor}

\begin{rem}\label{rem:avg-nonstationary-policies}
  The average-cost optimal policy referred to in
  Theorem~\ref{thm:optEqs} is in fact optimal over all
  randomized history-dependent policies; see e.g.,\ Hern\'{a}ndez-Lerma
  and Lasserre~\cite[Theorem~5.2.4]{hl-l1996}
\end{rem}
\begin{rem}
  Stationary average-cost optimal policies exist under much more
  general conditions than the ones considered in
  Theorem~\ref{thm:optEqs}. In particular, the Compactness Conditions
  and Assumption~HT imply Conditions (S) and (B) in
  Sch\"{a}l~\cite{schal1993}, as well as Assumptions (W$^*$) and
  (B) in Feinberg et al.~\cite{fkz2012}.
\end{rem}

\begin{rem}
  Under the hypotheses of Theorem~\ref{thm:optEqs}, the average-cost
  optimality equation (\ref{eq:acoe}) has a unique bounded solution up to an
  additive constant; see \cite[Lemma~3.3]{dekkerHordijk1992}. This is
  because Assumption~HT is a special case of the more general weighted
  geometric ergodicity condition considered in
  \cite{dekkerHordijk1992}; see \cite{dekkerEtal1994} for
  relationships between this condition and various other ergodicity
  and recurrence assumptions.
\end{rem}

\subsection{Finite state and action sets}
\label{sec:finite-state-action}

In this section, we assume that both $\X$ and $\A$ are finite. Recall
from the paragraph after Remark~\ref{rem:T-vs-HT} that, when the
state and action sets are finite, Assumption~HT is equivalent to the
existence of a constant $K^*$ such that
\begin{equation}\label{eq:T-finite-1}
\sum_{n = 0}^\infty {_\ell\mathbin{Q_\phi^n}} e(x) \leq K^* \qquad \text{for all}  \ \phi
\in \F, \ x
\in \X,
\end{equation}
where $e$ denotes the function on $\X$ that is identically equal to
one. Therefore, in this section we assume without loss of
generality that (\ref{eq:T-finite-1}) holds.

For a finite state and action MDP with transition probabilities $q$
that satisfy Assumption~HT, Corollary~\ref{cor:alg1} implies that a
stationary average-cost optimal policy can be computed by solving the
LP
\begin{equation}\label{eq:discLP}
  \begin{aligned}
    \text{minimize} &\quad \sum_{x \in \bar{\X}}\sum_{a \in
      \bar{A}(x)}\bar{c}(x, a) \bar{z}_{x,a} &\\
    \text{such that} &\quad \sum_{a \in \bar{A}(x)}\bar{z}_{x, a} -
    \bar{\beta}\sum_{y \in \bar{\X}}\sum_{a \in \bar{A}(y)}\bar{p}(x |
    y, a)\bar{z}_{y, a} = 1, & x \in \bar{\X}, \\
    &\quad \bar{z}_{x, a} \geq 0, & x \in \bar{\X}, \ a \in \bar{A}(x).
  \end{aligned}
\end{equation}
Recall that $m = \sum_{x \in \X}|A(x)|$ and $n = |\X|.$ If $K^*>1,$
it follows from Scherrer~\cite[Theorem~3]{scherrer2016} that the LP
(\ref{eq:discLP}) can be solved using 
\begin{displaymath}
  (m-n)\left\lceil \frac{1}{1 - \bar{\beta}} \log \frac{1}{1 -
      \bar{\beta}}\right\rceil = O((m-n)K^* \log K^*)
\end{displaymath}
iterations of the block-pivoting simplex method corresponding to
Howard's policy iteration algorithm.  In addition, 
it follows
from Scherrer~\cite[Theorem~4]{scherrer2016} that the
LP (\ref{eq:discLP}) can alternatively be solved using 
\begin{equation}
  \label{eq:39}
  n(m-n)\left(1 + \frac{2}{1 - \bar{\beta}} \log \frac{1}{1 -
      \bar{\beta}}\right) = O(n(m-n)K^* \log K^*)
\end{equation}
iterations of the simplex method with Dantzig's rule. Observe that $K^*=1$
means that the state $\ell$ is absorbing under each stationary policy,
and a stationary policy $\phi$ is average-cost optimal if and only if
$c(\ell,\phi(\ell))=\min\{c(\ell,a):a\in A(\ell)\}.$

\begin{rem}\label{Rem15}
  According to \cite[Proposition~12]{akianGaubert2013}, there is a
  one-to-one correspondence between sequences of policies generated by
  Howard's policy iteration algorithm for the discounted MDP defined
  by the HV-AG transformation, and sequences of policies generated by
  Howard's policy iteration algorithm for the original unichain
  average-cost MDP. In particular, under Howard's policy iteration
  algorithm for the discounted MDP, an improved policy $\phi^+$ is
  constructed (when possible) by taking, for each $x \in \X$,
  $\phi^+(x)$ to be any action belonging to
  \begin{equation}
    \label{eq:29}
    \argmin_{a \in A(x)}\left[ \bar{c}(x, a) + \bar{\beta}\sum_{y \in
        \X}\bar{p}(y | x, a)\bar{v}_{\bar{\beta}}^\phi(y)\right].
  \end{equation}
  Under Howard's policy iteration algorithm for unichain average-cost
  MDPs, given $\phi \in \F$ an improved policy $\phi^+$ is constructed
  by first obtaining a constant $g$ and a function $h$
  that satisfy the system of equations
  \begin{equation}
    \label{eq:31}
    g + h(x) = c(x, \phi(x)) + \sum_{y \in \X}q(y | x, \phi(x)) h(y), \qquad
    x \in \X,
  \end{equation}
  and then, for every $x \in \X$, taking $\phi^+(x)$ to be any action
  belonging to
  \begin{equation}
    \label{eq:30}
    \argmin_{a \in A(x)}\left[c(x, a) + \sum_{y \in \X}q(y | x, a)h(y)\right].
  \end{equation}
  Let
  $h^\phi(x) := \mu(x)[\bar{v}_{\bar{\beta}}^\phi(x) -
  \bar{v}_{\bar{\beta}}^\phi(\ell)]$ for $x \in \X$. According to
  Proposition~\ref{prp:1}, the constant
  $\bar{v}_{\bar{\beta}}^\phi(\ell)$ and the function $h^\phi$ satisfy
  (\ref{eq:31}). Further, the definitions of $\bar{c}$ and $\bar{p}$
  and Lemma~\ref{lem:1} imply that for each $x \in \X$ the set
  (\ref{eq:30}) is equal to the set (\ref{eq:29}). This implies that
  Howard's policy iteration algorithm for the discounted MDP defined
  in Section~\ref{sec:hv-ag-transformation} is equivalent to a
  particular version of Howard's policy iteration algorithm for the
  original unichain average-cost MDP.
  Since both of these policy iteration algorithms correspond to
  block-pivoting simplex methods (see \cite[pp.\ 68, 122]{kallenberg1983}, it
  follows from Scherrer \cite[Theorem~3]{scherrer2016} that, when
  there is a state that is recurrent under all stationary policies, the well-known LP for
  unichain average-cost MDPs, see e.g., \cite[LP 4.6.7]{kallenberg1983},
   \begin{equation}\label{eq:LP-ACStand}
  \begin{aligned}
    \text{minimize} &\quad \sum_{x \in {\X}}\sum_{a \in
      {A}(x)}{c}(x, a) {z}_{x,a}
    \\
    &\quad \sum_{a \in {A}(x)}{z}_{x, a} -
    \sum_{y \in {\X}}\sum_{a \in {A}(y)}q(x |
    y, a){z}_{y, a} = 0, & x \in {\X}, \\
    &\quad    \sum_{y \in {\X}}\sum_{a \in {A}(y)}{z}_{y, a} = 1, & x \in {\X},  \\
     &\quad {z}_{x, a} \geq 0, & x \in {\X}, \ a \in {A}(x),
  \end{aligned}
\end{equation}
   can
  be solved using $O((m-n)K^* \log K^*)$ iterations of a block-pivoting
  simplex method.

\end{rem}

\begin{rem}\label{rem:computing-Kstar}
  For $x \in \X$, let
  $\tau_\ell(x) := \sup_{\phi \in \F}\sum_{n=0}^\infty
  {_\ell\mathbin{Q_\phi^n}} e(x)$. Then
  $K_\ell := \max_{x \in \X}\tau_\ell(x)$ is the smallest constant
  $K^*$ satisfying (\ref{eq:T-finite-1}). The iteration estimate for
  Howard's policy iteration algorithm for average-cost MDPs satisfying
  (\ref{eq:T-finite-1}) that follows from Akian and
  Gaubert~\cite[Corollary~15]{akianGaubert2013} is
  $O((m-n)K_\ell \log K_\ell)$. One method to compute $K_\ell$ consists of
  the following. First, compute an optimal policy $\phi_*$ for a
  transient MDP that is identical to the original MDP, except that
  state $\ell$ is removed and all one-step costs are equal to $-1$.
  Then, compute the value function $v^{\phi_*}$ of this optimal
  policy, set
  \begin{displaymath}
    v^{\phi_*}(\ell) := \max_{a \in A(x)}\left[1 + \sum_{y \neq
        \ell}q(y | \ell, a)v^{\phi_*}(y)\right],
  \end{displaymath}
  and set $K_\ell = \max_{x \in \X}v^{\phi_*}(x)$. According to
  Denardo~\cite[Theorem~2]{denardo2016}, the policy $\phi_*$ can be
  computed using $O((m-n)K_\ell \log K_\ell)$ iterations of Howard's
  policy iteration algorithm. Further, the function $v^{\phi_*}$ can
  be computed by solving a system of $n-1$ linear equations, using
  Gaussian elimination in $O(n^3)$ arithmetic operations; for other
  methods see e.g., \cite{strassen1969}, \cite{c-w1990}. 
\end{rem}

\begin{rem}
  Applying the simplex method with Dantzig's rule to the LP (\ref{eq:discLP}) 
  can be viewed as applying a certain pivoting rule to the LP
  (\ref{eq:LP-ACStand}). In particular, for $\phi \in \F$ let
  $h^\phi(x) := \mu(x)[\bar{v}_{\bar{\beta}}^\phi(x) -
  \bar{v}_{\bar{\beta}}^\phi(\ell)]$ for $x \in \X$.  Given a
  non-optimal basic feasible solution to (\ref{eq:LP-ACStand})
  corresponding to the non-optimal stationary policy $\phi$, it
  follows from Lemma~\ref{lem:1} and Proposition~\ref{prp:1} that the
  variable $z_{x, a}$ that enters the basis under this pivoting rule
  is the one minimizing
  \begin{equation}
    \label{eq:35}
    \bar{c}(x, a) + \bar{\beta}\sum_{y \in \X}\bar{p}(y | x,
    a)\bar{v}_{\bar{\beta}}^\phi(y)  - \bar{v}_{\bar{\beta}}^\phi(x) =
    \frac{1}{\mu(x)}\left[c(x, a) + \sum_{y \in \X}q(y | x,
      a)h^\phi(y) - w^\phi - h^\phi(x)\right],
  \end{equation}
  and the variable that leaves the basis is $z_{x \phi(x)}$. 
  According to (\ref{eq:39}), this
  pivoting rule for the LP (\ref{eq:LP-ACStand}), that is typically 
  used to solve unichain average-cost MDPs, is strongly polynomial when
  $K^*$ is fixed. This algorithm is not the same as applying Dantzig's
  rule to the LP (\ref{eq:LP-ACStand}), however; see
  Remark~\ref{rem:discAvgLP}.
\end{rem}

\begin{rem}\label{rem:discAvgLP}
  Since an MDP satisfying Assumption~HT is unichain, an optimal policy
  under the average-cost criterion can be computed by solving the LP~\eqref{eq:LP-ACStand};
  see e.g.,
  \cite[LP~4.6.7]{kallenberg1983}. As follows from Remark~\ref{Rem15}, under Assumption HT, starting with the same basic variables, the sequences of basic variables for implementations of block-pivoting simplex methods for
  the LPs~\eqref{eq:discLP} and \eqref{eq:LP-ACStand} coincide.  However, this is not true for the
  the simplex method with Dantzig's rule. To confirm this, let us consider the following example.  The set of
  states is $\X = \{1, 2\}$ and the sets of available actions are
  $A(1) = A(2) = \{a, b\}$. The transition probabilities form
  stochastic vectors given by $p(1 | 1, a) = 1/2$, $p(1 | 1, b) = 0$,
  $p(1 | 2, a) = 1/3$, and $p(1 | 2, b) = 1/2$. The one-step costs are
  $c(1, a) = c(1, b) = 1$ and $c(2, a) = c(2, b) = 2$. Letting
  $\ell = 1$, one can verify that the function $\mu$ defined by
  $\mu(1) = 10$ and $\mu(2) = 3$ satisfies (\ref{eq:1}) with
  $V \equiv 1$.  The average-cost LP given by the LP (\ref{eq:LP-ACStand}) is
  \begin{equation}\label{eq:LP-AV}
  \begin{aligned}
    \text{minimize} &\quad z_{1, a} + z_{1, b} + 2z_{2, a} + 2z_{2,b}
    \\
    \text{such that} &\quad \frac{1}{2}z_{1, a} + z_{1, b} -
    \frac{1}{3}z_{2, a} - \frac{1}{2}z_{2, b} = 0\\
    &\quad -
    \frac{1}{2}z_{1, a} - z_{1, b}  + \frac{1}{3}z_{2, a} + \frac{1}{2}z_{2, b} = 0\\
    &\quad z_{1, a} + z_{1, b} + z_{2, a} + z_{2, b} = 1 \\
    &\quad z_{1, a}, z_{1, b}, z_{2, a}, z_{2, b} \geq 0,
  \end{aligned}
\end{equation}
and the LP (\ref{eq:discLP}) for the discounted MDP defined by the
HV-AG transformation is
\begin{equation}\label{eq:LP-DISC}
  \begin{aligned}
    \text{minimize} &\quad \frac{1}{10}z_{1, a} + \frac{1}{10}z_{1, b}
    + \frac{1}{3}z_{2, a} + \frac{1}{3}z_{2, b}\\
    \text{such that} &\quad \frac{1}{4}z_{1, a} + \frac{2}{5}z_{1, b}
    - \frac{1}{6}z_{2, b} = 1 \\
    &\quad -\frac{3}{20}z_{1, a} - \frac{3}{10}z_{1, b} +
    \frac{1}{3}z_{2, a} + \frac{1}{2}z_{2, b} = 1 \\
    &\quad z_{1, a}, z_{1, b}, z_{2, a}, z_{2, b} \geq 0.
  \end{aligned}
\end{equation}
For both LPs, suppose the initial basic feasible solution for the
simplex method with Dantzig's rule is the one defined by the
stationary policy $\phi$ where $\phi(1) = b$ and $\phi(2) = a$;
namely, the basic variables are $z_{1, b}$ and $z_{2, a}$. Consider
the first iteration of this simplex method. For the LP
(\ref{eq:LP-AV}), the basic variable $z_{1, b}$ is the unique variable
to leave the basis, while for the LP (\ref{eq:LP-DISC}) the basic
variable $z_{2, a}$ is the unique variable to leave the basis.
\end{rem}

\begin{rem}
  Consider an LP with $n$ constraints and $m$ variables, where the
  positive elements of every basic feasible solution are bounded below
  by $\delta$ and bounded above by $\gamma$. By generalizing the
  analysis in Ye~\cite{ye2011} for discounted MDPs, it is proved in
  Kitahara and Mizuno~\cite[Theorem~3]{kitaharaMizuno2013} that the
  simplex method with Dantzig's rule requires at most
  \begin{displaymath}
    O\left(nm\frac{\gamma}{\delta} \log \frac{\gamma}{\delta}\right)
  \end{displaymath}
  iterations to return an optimal solution. For the LP
  (\ref{eq:discLP}), $\delta = 1$ and
  $\gamma = (1 - \tilde{\beta})^{-1} = K^*$ satisfy the hypotheses of
  this result. Therefore, it follows from
  \cite[Theorem~3]{kitaharaMizuno2013} that an average-cost optimal
  policy can be computed in strongly polynomial time when $K^*$ is
  fixed, by applying the simplex method with Dantzig's rule to the LP
  (\ref{eq:discLP}). However, \cite[Theorem~3]{kitaharaMizuno2013}
  does not imply an analogous statement for the LP
  (\ref{eq:LP-ACStand}) for unichain average-cost MDPs. This is
  because, for such MDPs, every basic feasible solution of
  (\ref{eq:LP-ACStand}) is the vector of state-action frequencies
  under some stationary policy
  \cite[Remark~4.7.4]{kallenberg1983}. Even for MDPs satisfying
  Assumption~HT with a fixed $K^*$, these frequencies can decrease
  exponentially with the number of states. To verify this, for
  $n = 2,3, \dots$ consider an MDP with state set
  $\X := \{1, \dots, n\}$, a single action 0 available at every state,
  transition probabilities
  $p(1 | 1, 0) = p(n | i, 0) = p(i | i+1, 0) := 1/2$ for
  $i = 1, \dots, n-1$, and arbitrary real-valued one-step
  costs. Observe that for $n = 1, 2, \dots$, this MDP satisfies
  Assumption~HT with $\ell = n$ and $K^* = 2$. In addition, the unique
  feasible solution to (\ref{eq:LP-ACStand}) for this MDP is
  \begin{displaymath}
    z_{1, 0} = \left(\frac{1}{2}\right)^{n-1}, \quad z_{i, 0} =
    \left(\frac{1}{2}\right)^{n-i+1}, \ \text{for} \ i=2, \dots, n.
  \end{displaymath}
  Thus, there is no $\delta>0$ such that $z_{1,0}\ge\delta$ for all $n=2,3,\ldots\ .$
\end{rem}

\subsection{Extension to uncountable state spaces}
\label{sec:extens-unco-state-1}

For $\phi \in \F$, let $\mathbin{_\ell Q_\phi}$ be defined for an
integrable Borel-measurable $u:\X \rightarrow \R$ as
\begin{displaymath}
  \mathbin{_\ell Q_\phi}u(x) := \int_{\X \setminus \{\ell\}}u(y)q(dy |
  x, \phi(x)), \qquad x \in \X.
\end{displaymath}
The version of
Assumption~HT that we consider when the state space is possibly
uncountable is as follows:
\begin{samepage}
  \begin{asmpHTprime} \
    \begin{enumerate}[(i)]
    \item There is a state $\ell \in
      \X$ 
      and a constant $K^*$ satisfying
      \begin{equation}\label{eq:HT-1}
        \| \sum_{n = 0}^\infty \mathbin{_\ell Q_\phi^n}\| \leq K^* < \infty
        \quad \text{for all} \ \phi \in \F.
      \end{equation}
    \item The one-step cost function $c$ is bounded.
    \end{enumerate}
\end{asmpHTprime}
  \end{samepage}

\subsubsection{HV-AG transformation}
\label{sec:hv-ag-transformation-1}

Suppose Assumption~HT' holds. According to Proposition~\ref{prp:SH-1},
there is a Borel-measurable function $\mu:\X \rightarrow [1, \infty)$
that satisfies $\mu \leq K^*$ and
\begin{equation}\label{eq:mu-HVAG}
  \mu(x) \geq 1 + \int_{\X \setminus \{\ell\}}\mu(y)q(dy | x, a),
  \qquad (x, a) \in \Gr(A).
\end{equation}

Here the HV-AG transformation is defined exactly as described in
Section~\ref{sec:hv-ag-transformation}, except that the cost-free
absorbing state $\bar{x}$ is taken to be isolated from $\X$, and the
transition probabilities $\bar{p}$ are defined by
\begin{displaymath}
  \bar{p}(B | x, a) :=
  \begin{cases}
    \frac{1}{\bar{\beta}\mu(x)}\int_B \mu(y)q(dy | x, a), &\quad B \in
    \B(\X \setminus \{\ell\}), \ (x, a) \in \Gr(A), \\
    \frac{1}{\bar{\beta}\mu(x)}[\mu(x) - 1 - \int_{\X \setminus
      \{\ell\}}\mu(y)q(dy | x, a)], &\quad B = \{\ell\}, \ (x, a) \in
    \Gr(A), \\
    1 - \frac{1}{\bar{\beta}\mu(x)}[\mu(x) - 1], &\quad B =
    \{\bar{x}\}, \ (x, a) \in \Gr(A), \\
    1, &\quad B = \{\bar{x}\}, \ (x, a) = (\bar{x}, \bar{a}).
  \end{cases}
\end{displaymath}

\subsubsection{Results}
\label{sec:results-3}

\begin{lem}\label{lem:1-1}
  If a bounded Borel function $f: \bar{\X} \rightarrow \R$ satisfies
  $f(\bar{x}) = 0$, then for any $x \in \X$ and $a \in A(x)$
  \begin{equation}\label{eq:mainLem-1}
    \bar{c}(x, a) + \bar{\beta}\int_{\bar{\X}}f(y)\bar{p}(dy | x, a) =
    \frac{1}{\mu(x)}\left[c(x, a) + \int_\X \mu(y)[f(y) - f(\ell)] q(dy | x,
      a) + [\mu(x) - 1]f(\ell)\right].
  \end{equation}
\end{lem}
\begin{proof}
  This follows from the proof of Lemma~\ref{lem:1}, with all sums
  replaced with integrals.
\end{proof}

\begin{prp}\label{prp:1-1}
  Suppose Assumption~HT' holds. Let $\phi \in \F$ be a stationary policy and
  $h^\phi(x) := \mu(x)[\bar{v}_{\bar{\beta}}^\phi(x) -
  \bar{v}_{\bar{\beta}}^\phi(\ell)]$ for $x \in \X$. Then
  \begin{equation}\label{eq:correspPol-1}
    \bar{v}_{\bar{\beta}}^\phi(\ell) + h^\phi(x) =
    c(x, \phi(x)) + \int_\X h^\phi(y) q(dy | x,
    \phi(x)), \quad x \in \X.
  \end{equation}
  In addition, if the transition rates $q$ are stochastic,
  then $w^\phi \equiv \bar{v}_{\bar{\beta}}^\phi(\ell)$.
\end{prp}
\begin{proof}
  This follows from the proof of Proposition~\ref{prp:1}, where sums
  are replaced with integrals in the appropriate places.
\end{proof}

\begin{lem}\label{lem:asmpS-HVAG}
  Suppose Assumptions~S and HT' hold. Then the discounted MDP defined
  by the HV-AG transformation also satisfies Assumption~S.
\end{lem}
\begin{proof}
  This follows from Lemma~\ref{lem:mu-integral-cont-1} and the proof
  of Lemma~\ref{lem:compactnessHVAG}.
\end{proof}

To state the main result in this section, for $x \in \X$ define
$A_{\text{av}}^*(x)$ by replacing the sum in (\ref{eq:A-star-av}) with
an integral.
\begin{samepage}
\begin{thm} Suppose the original MDP with transition probabilities $q$ satisfies
  Assumptions~S and HT'. Then:
  \begin{enumerate}[(i)]
  \item the constant $w = \bar{v}_{\bar{\beta}}(\ell)$ and the
    function
    $h(x) = \mu(x)[\bar{v}_{\bar{\beta}}(x) -
    \bar{v}_{\bar{\beta}}(\ell)]$, $x \in \X$, satisfy the optimality
    equation
  \begin{equation}\label{eq:acoe-1}
    w + h(x) = \min_{A(x)}\left[c(x, a) + \int_\X h(y)q(dy | x,
      a)\right], \quad x \in \X,
  \end{equation}
  and $\bar{v}_{\bar{\beta}}(\ell)$ is the optimal average cost for
  each initial state.
\item there is a $\phi \in \F$ satisfying $\phi(x) \in A_{\text{av}}^*(x)$
  for all $x \in \X$, where
  \begin{equation}
    \label{eq:14-1}
    A_{\text{av}}^*(x) = \left\{ a \in A(x) \ \left\vert\vphantom{\frac{1}{1}}\right. \ \bar{v}_{\bar{\beta}}(x) = \bar{c}(x, a) +
        \bar{\beta}\int_{\bar{\X}}\bar{v}_{\bar{\beta}}(y)\bar{p}(dy | x, a) \right\}, \quad x \in \X,
  \end{equation}
  and any such policy is average-cost optimal.
  \end{enumerate}
\end{thm}
\end{samepage}
\begin{proof}
  This follows from the proof of Theorem~\ref{thm:optEqs}, where sums
  are replaced with integrals in the appropriate places.
\end{proof}

\bigskip

{\bf Acknowledgements.}  The authors thank the associate editor, two
anonymous referees, Rolando Cavazos-Cadena, Eric Denardo, and Pavlo
Kasyanov for useful comments and suggestions. This research was
partially supported by NSF grants CMMI-1335296 and and CMMI-1636193.




\end{document}